\documentclass{elsarticle}
\usepackage{amsmath}
\usepackage{amsthm}
\usepackage{amssymb}
\usepackage{amsfonts}
\usepackage{epsfig}
\usepackage{setspace}
\usepackage{color}

\newtheorem{theorem}{Theorem}[section]

\newtheorem{teo}[theorem]{Theorem}
\newtheorem{lm}[theorem]{Lemma}

\newtheorem{cor}[theorem]{Corollary}
\newtheorem{df}[theorem]{Definition}
\newtheorem{rem}[theorem]{Remark}
\newtheorem{pr}[theorem]{Proposition}

\newcommand{\la}{\lambda}

\DeclareMathOperator{\Hom}{Hom}
\DeclareMathOperator{\de}{det}
\DeclareMathOperator{\GL}{GL}
\DeclareMathOperator{\Dist}{Dist}

\begin{document}

\begin{frontmatter}

\title{Primitive vectors in induced supermodules for general linear supergroups}

\author{Franti\v sek~ Marko}

\ead{fxm13@psu.edu}

\address{Penn State Hazleton, 76 University Drive, Hazleton, PA 18202, USA}

\begin{abstract} The purpose of the paper is to derive formulas that describe the structure of the induced supermodule $H^0_G(\la)$ 
for the general linear supergroup $G=GL(m|n)$ over an algebraically closed field $K$ of characteristic $p\neq 2$. Using these 
formulas we determine primitive $G_{ev}=GL(m)\times GL(n)$-vectors in $H^0_G(\lambda)$. We conclude with remarks related to the linkage 
principle in positive characteristic.
\end{abstract}

\begin{keyword}
general linear supergroup \sep  primitive vectors \sep linkage principle
\MSC 15A15 \sep 17A70 \sep 20G05 \sep 15A72 \sep 13A50 \sep 05E15 
\end{keyword}


\end{frontmatter}

\section*{Introduction}

The concept of the determinant is one of the capstone of linear algebra with applications throughout most areas of mathematics.
A special type of products of determinants, called bideterminants, corresponding to (Young) tableaux appeared in the invariant theory, 
representation theory of the symmetric group and the general linear group.  The main reason of the success of bideterminants is the 
straightening formula which determine that every bideterminant is an integral linear combination of semistandard bideterminants 
(which in turn are linearly independent). The main tool in the proof of the straightening formula is the Laplace duality.
The Laplace duality relies heavily on the combinatorics of tableaux of different shapes and the only linear algebra technique used 
is the Laplace expansion of determinants. For the survey of the role of bideterminants in the representation theory see
\cite{fm3}. 

On the other hand, there are classical determinantal identities of Gauss, Schur, Cauchy, Sylvester, Kronecker, Jacobi, Binet, Laplace, Muir, Cayley
and others (see \cite{bs}) that relate various determinants and should be therefore useful in the above context of bideterminants. In this paper we will 
show a direct application of determinantal identities of the minors of a matrix, like the Jacobi theorem on minors of the adjoint matrix, and of 
the Muir's ``law of extensible minors'' (to derive new determinantal identity from another determinantal identity) to Schur superalgebras and general linear 
supergroups. These determinantal identities are crucial tool in the description of the algebraic structure of the induced supermodules (the main building blocks)
of the general linear supergroup.  

To give an example of determinantal identities we will use, consider an $m\times m$ matrix
$C_{11}= (c_{ij})_{i,j=1, \ldots m}$, denote by 
\[A=\left(%
\begin{array}{cccc}
A_{11} & A_{12} & \ldots & A_{1m} \\ 
A_{21} & A_{22} & \ldots & A_{2m} \\ 
\ldots & \ldots & \ldots & \ldots \\ 
A_{m1} & A_{m2} & \ldots & A_{mm}%
\end{array}%
\right)\] 
the adjoint of the matrix $C_{11}$ and denote by $D$ the determinant of $C_{11}$.
Denote by $C(i_{1},\ldots ,i_{t}|j_{1},\ldots j_{t})$ the $t$-th
minor of $C$ corresponding to the rows $\{i_{1},\ldots ,i_{t}\}$ and columns 
$\{j_{1},\ldots ,j_{t}\}$ of matrix $C$.
Fix $1\leq k_1<k_2\ldots <k_j\leq m$ and $0\leq j<m$. Then 
\[\sum_{a=1}^{m} C(1, \ldots, j+1|k_1,\ldots, k_j, a) A_{a,s}= (-1)^{s+j+1}C(1,\ldots, \hat{s}, \ldots, j+1|k_1, \ldots, k_j)D\]
if $s\leq j+1$, and zero otherwise.
This identity is established in the proof of Lemma \ref{new} and later used to determine the action of general linear supergroup on its induced supermodules. 

Throughout the paper we will be working over an algebraically closed field $K$ of characteristic $p\neq 2$.
In the next section we will define the general linear supergroup $G=GL(m|n)$, its coordinate ring $K[G]$ and the even subgroup $G_{ev}$ of $G$, which is isomorphic to the product 
$\GL(m)\times \GL(n)$ of corresponding general linear groups.
The focus of our investigation are induced supermodules $H^0_G(\la)$ which, in the next section, are explicitly identified  with certain subsupermodules of $K[G]$.
This is very fortunate since it allows us to bypass the full complexity of the $G$-supermodule category and instead we can work inside $K[G]$, the $G$-supermodule 
structure of which is obtained from right superderivations. 
The induced module $H^0_{G_{ev}}(\la)$ has a basis given by bideterminants. In order to describe the structure of $H^0_G(\la)$, it is  
important to understand the action of odd superderivations on these bideterminants from $H^0_{G_{ev}}(\la)$. 
It is this crucial step where we need to apply previously mentioned determinantal identities.  The corresponding 
formulas and other basic properties are derived in Section 2. The main result of Section 2 is Theorem \ref{gen}.

Let $B$ be the Borel subsupergroup of $G$ 
corresponding to lower triangular matrices and $B_{ev}$ be the intersection of $B$ with $G_{ev}$. A vector $w$ is called a primitive $G$-vector if 
the one-dimensional subspace spanned by $w$ is stabilized by $B$ and is called a primitive $G_{ev}$-vector if this subspace is stabilized by $B_{ev}$.
The primitive vectors are important because the categories of $G$-supermodules and $G_{ev}$-modules are highest weight categories and all their irreducible 
(super)modules are generated by a highest weight vector corresponding to a certain weight $\la$, and these highest vectors are primitive. 
If the characteristic $p$ of the base field $K$ is zero, then the category of $G_{ev}$-modules is semisimple and completely described by simple modules 
corresponding to primitive $G_{ev}$-vectors. If the characteristic $p>0$, then $G_{ev}$-modules are not semisimple in general, but primitive vectors still 
describe all simple $G_{ev}$-modules.
The category of $G$-supermodules is not semisimple even when $p=0$. It is well known that if $p=0$, then the induced supermodule $H^0_G(\la)$ 
is irreducible if and only if the highest weight $\la$ is typical in the sense of Kac \cite{kac3}. The characterization of irreducible $H^0_G(\la)$ in the case $p>2$ 
was given in \cite{fm2}.
The characters of simple supermodules for the special cases when $G=G(3|1)$ and $G=G(2|2)$ were determined in \cite{fm} and \cite{gm1}. 
A very important consideration in these papers was the determination of the $G_{ev}$-module structure of $H^0_G(\la)$. It was possible only after the 
primitive $G_{ev}$-vectors were determined.

In the second part of this paper we investigate primitive $G_{ev}$-vectors in $H^0_G(\la)$ for $G=\GL(m|n)$ in the case of characteristic zero.
In this case the multiplicities of primitive $G_{ev}$-vectors of a fixed weight in an irreducible $G$-module are described by Littlewood-Richardson coefficients but it is not clear how to describe a basis of these primitive $G_{ev}$-vectors.  
We will determine the $G_{ev}$-primitive vectors that belong to the first floor $F_1$ of $H^0_G(\la)$. Also, in the case when the multiplicity of primitive $G_{ev}$-vectors of a given weight is maximal (in the sense explained in Section \ref{s4}), the primitive vectors are completely determined in Theorem \ref{Fwedge}. We also show that this procedure of constructing $G_{ev}$-primitive vectors might be useful even when the above mentioned multiplicity is not maximal.  
This way we obtain a better understanding of the $G_{ev}$-structure of $H^0_G(\la)$, and also provide a step toward understanding of the $G_{ev}$-structure of $H^0_G(\la)$ in arbitrary characteristic.

Based on the above results, in the last part of the paper addressed in Section 5, we are able to present 
certain results on blocks and linkage principle for $G$ over the field of characteristic $p>2$. 
In Theorem \ref{link1} we show that if $H^0_{G_{ev}}(\la)$ is irreducible, then simple composition factors of $H^0_G(\la)$ belonging to its first floor $F_1$ are determined using quantities appearing in the definition of the typical weight. 
Additionally, we consider $G_{ev}$-linkage and a relation which we call odd linkage. A combination of both linkages gives a relation analogous to 
to the linkage principle in the characteristic zero case. 

\section{Definitions and notation}

\subsection{General linear supergroup}

Good reference for the basic properties of general linear supergroups, induced modules and superderivations is \cite{z} and \cite{zs}. The reader is advised to consult these papers for further details. 

In order to define the general linear supergroup $G=\GL(m|n)$, start with 
a generic $(m+n)\times (m+n)$-matrix $C=(c_{ij})$ as a block matrix 
\begin{equation*}
C=%
\begin{pmatrix}
C_{11} & C_{12} \\ 
C_{21} & C_{22}%
\end{pmatrix}%
,
\end{equation*}
where $C_{11}, C_{12}, C_{21}$ and $C_{22}$ are matrices of sizes $m\times m$%
, $m\times n$, $n\times m$ and $n\times n$, respectively. 
A superalgebra is a $\mathbb{Z}_2$-graded space in which the parity of the product of two homogeneous elements is the sum of their parities. 
Define the parity $|i|$ of symbol $i$ by $|i|=0 \pmod 2$ for $i=1, \ldots, m$ and 
$|i|=1 \pmod 2$ for $i=m+1, \ldots, m+n$, and the parity $|c_{ij}|$ of $c_{ij}$ as $|i|+|j| \pmod 2$.
The commutative superalgebra $A(m|n)$ is freely generated by elements $c_{ij}$ for $1\leq i,j \leq m+n$
subject to the supercommutativity relation
$c_{ij}c_{kl}=(-1)^{|c_{ij}||c_{kl}|} c_{kl}c_{ij}$.
The coordinate ring $K[G]$ of $G$ is defined as a localization of $A(m|n)$ at the element $DD_{22}$,
where $D= \de(C_{11})$ and $D_{22}= \de(C_{22})$. The ring $K[G]$ has a structure of a Hopf superalgebra.
The general linear supergroup $G=\GL(m|n)$ is a functor $\Hom_{\rm superalg}(K[G], -)$ 
(of morphisms preserving parity of elements) that assigns to every commutative superalgebra $S=S_0\oplus S_1$ (split into its even and odd parts)
the supergroup $GL(m|n)(S)$ 
of invertible $(m+n)\times (m+n)$-matrices
\begin{equation*}
\begin{pmatrix}
S_{11} & S_{12} \\ 
S_{21} & S_{22}%
\end{pmatrix}%
,
\end{equation*}
where $\GL(m)$ and $\GL(n)$ are classical general linear groups and blocks $S_{11}\in GL(m)(S_0), S_{22}\in GL(n)(S_0)$ have coefficients from $S_0$ 
and blocks $S_{12}$ and $S_{21}$ have coefficients from $S_1$.

Of particular importance to us will be the even subgroup $G_{ev}$ of $G$, which is isomorphic to $\GL(m)\times \GL(n)$, 
corresponding to invertible matrices of type 
\begin{equation*}
\begin{pmatrix}
C_{11} & 0 \\ 
0& C_{22}%
\end{pmatrix}%
.\end{equation*}

Let $T$ be the maximal torus of $G$ corresponding to diagonal matrices and $B$ be the lower triangular Borel subsupergroup of $G$. 
Further, denote by $P$ the parabolic subsupergroup of $G$ corresponding to $G_{ev}$ and $B$, consisting of matrices
\begin{equation*}
\begin{pmatrix}
C_{11} & 0 \\ 
C_{21} & C_{22}%
\end{pmatrix}%
\end{equation*}
and denote by $U$ the odd unipotent subgroup of $G$ consisting of matrices
\begin{equation*}
\begin{pmatrix}
E_m & C_{12} \\ 
0 & E_n%
\end{pmatrix}%
,
\end{equation*}
where $E_m$ and $E_n$ denote unit matrices.

Let $V$ be a superspace. The left $G$-supermodule structure on $V$ corresponds uniquely to a right
$K[G]$-supercomodule on $V$. 
Using Sweedler's notation, write the supercomodule map $\tau_V$ as $\tau_V(v)=\sum v_1\otimes h_2$ for $v_1\in V$ and $h_2\in K[G]$.
Then the action of $G$ on $V$ is the family of (functorially compatible) actions of $G(A)$ on $V\otimes A$ 
defined by the rule 
\[g\cdot (v\otimes d)=\sum v_1\otimes g(h_2)d,\]
where $g\in G(A), d\in A,$ and $A$ runs over all commutative superalgebras.

Every $T$-supermodule $V$ decomposes into a direct sum of weight supersubspaces $V_{\lambda}$, where $\lambda =
(\lambda_1, . . . , \lambda_{m+n})\in X(T) = \mathbb{Z}^{m+n}$ are called weights of $V$. 
Here $X(T)$ is abelian group of characters and is freely generated by $\epsilon_1, \ldots, \epsilon_{m+n}$, where $\epsilon_{i}$ picks the $i$-th 
entry in a diagonal matrix from $T$. By definition, $v\in V_{\lambda}$ if and only if 
$\tau_{V|_{T}} (v)= v\otimes \prod_{i=1}^{m+n} c_{ii}^{\lambda_i}$.

Since $t\cdot c_{ij}=c_{ij}\otimes t_j$ for $t\in T(A)$, the weight of $c_{ij}$ is $\epsilon_j$. 
Let $I=(i_1, \ldots, i_r)$ and $J=(j_1, \ldots, j_r)$ be multi-indices with entries in the set $\{1, \ldots, m+n\}$. Define the content of $J$ by $cont(J)=(x_1, \ldots, x_{m+n})$, 
where $x_i$ is the number of occurences of the integer $i$ in $J$. 
Then the weight of the product $c_{I|J}=c_{i_1j_1} \ldots c_{i_rj_r}$ is $cont(J)$.

We will also write the weight $\la$ as \[\la=(\la_1, \ldots, \la_m|\la_{m+1}, \ldots, \la_{m+n}),\]
and \[\lambda=(\lambda^+|\lambda^-)=(\lambda^+_1,
\ldots,\lambda^+_m |\lambda^-_{1}, \ldots, \lambda^-_{n})\] 
and use both notations interchangeably. 

Assume from now on that $\lambda=(\lambda^+|\lambda^-)$
is a dominant integral weight of $G$, that is $\lambda^+_1\geq \ldots \geq \lambda^+_m$ and $%
\lambda^-_{1} \geq \ldots \lambda^-_{n}$ and all $\lambda^{\pm}_i$ are integers. 

The following definition of typical and atypical weights plays an important role in the representation theory of $GL(m|n)$. This definition agrees
with the previously mentioned notion of typicality and atypicality due to Kac \cite{kac3}.

\begin{df} 
For $i=1,\ldots, m$ and $j=1,\ldots, n$ define $\omega_{ij}(\lambda)=\lambda^+_i+\lambda^-_j+m+1-i-j$ and call
the weight $\lambda$ \emph{atypical} if there exist $i\in\{1, \ldots, m\}$ and $j\in\{1,\ldots, n\}$ such that $\omega_{ij}=0$ provided $char K=0$, or 
$\omega_{ij}\equiv 0 \pmod p$ if $char K=p>0$; 
otherwise call the weight $\lambda$ \emph{typical}.
\end{df}

In order to study the structure of the $G$-supermodules, we will use the superalgebra of distributions $\Dist(G)$ of $G$ described in Section 3 of \cite{bk}.
Let $\mathfrak{m}$ be the kernel of the augmentation map $\epsilon$ of the Hopf superalgebra $K[G]$, and 
$Dist_r(G)=(K[G]/\mathfrak{m}^{r+1})^*$, where $*$ is the duality $\Hom_K(-,K)$ for each $r\geq 0$.
Then $Dist(G)=\bigcup_{r\geq 0} Dist_r(G)$. The left action of $G$ on $V$ as above induces a left action of $Dist(G)$ on $V$
by $\phi\cdot v=\sum v_1\phi(h_2)$. Conversely, any integrable $Dist(G)$ module can be lifted in a unique way to $G$ (see section 3 of \cite{bk}).
Let us note that the action $\phi\cdot v$ is often defined in accordance with the rule of signs as $\phi\cdot v=\sum (-1)^{|\phi||v_1|} v_1\phi(h_2)$. As explained on p.100 of 
\cite{mz}, there is an equivalence of corresponding categories of supermodules. We will work with our definition since the computations become simpler that way.

Denote by $e_{ij}$ the elements of $\Dist_1(G)$ determined by $e_{ij}(c_{hk}) = \delta_{ih}\delta_{jk}$ and $e_{ij}(1) = 0$.  
Then the parity of $e_{ij}$ is a sum of parities $|i|$ of $i$ and $|j|$ of $j$.
Then $e_{ij}$ belong to the Lie superalgebra $Lie(G)=(\mathfrak{m}/\mathfrak{m}^2)^*$ which is identified with the general 
linear Lie superalgebra $\mathfrak{gl}(m|n)$. Under this identification, $e_{ij}$ corresponds to the matrix unit which has all entries zeroes except the
entry at the position $(i,j)$ which equals one.
The commutation relations for the matrix units $e_{ij}$ are given as
\[[e_{ab},e_{cd}]=e_{ad}\delta_{bc}+(-1)^{(|a|+|b|)(|c|+|d|)} e_{cb}\delta_{ad}.\] 

Over the field of complex numbers $\mathbb{C}$, the distribution superalgebra $Dist (G)$ is isomorphic to the universal enveloping superalgebra $U_{\mathbb{C}}$ of $\mathfrak{gl}(m|n)$. 
To describe $Dist(G)$ in general, one considers the Kostant $\mathbb{Z}$-form $U_\mathbb{Z}$ is generated by elements $e_{ij}$ for odd $e_{ij}$, 
$e_{ij}^{(r)}= \frac{e_{ij}^r}{r!}$ for even $e_{ij}$, and 
$\binom{e_{ii}}{r}=\frac{e_{ii}(e_{ii}-1)\ldots (e_{ii}-r+1)}{r!}$ for all $r>0$.
Then $Dist(G)$ over the field $K$ is isomorphic to the Hopf superalgebra $K\otimes_{\mathbb{Z}} U_{\mathbb{Z}}$.

Finally, we will represent the natural (left) action of $Dist(G)$ on $A(m|n)$ using (right) superderivations of $A(m|n)$.
A right superderivation $_{ij}D$ of $A(m|n)$ of parity $|i|+|j|\pmod 2$ satisfies
\[(ab)_{ij}D=(-1)^{(|i|+|j|)|b|} (a)_{ij}Db+a(b)_{ij}D\] 
for $a,b\in A(m|n)$ and is given by \[(c_{kl})_{ij}D=\delta_{li}c_{kj}.\]

The superalgebra $A(m|n)$ is a right supercomodule with respect to the comultiplication
\[c_{ij}\mapsto \sum_{1\leq k\leq m+n} c_{ik}\otimes c_{kj}.\] 
Therefore any $g\in G(A)$, where $A$ is a commutative superalgebra, acts on $c_{ij}$ by the rule
\[g\cdot c_{ij}\mapsto \sum_{1\leq k\leq m+n} c_{ik}\otimes g(c_{kj})= \sum_{1\leq k\leq m+n} c_{ik}\otimes g_{kj}.\]
In other words, if we identify $c_{ij}$ with $c_{ij}\otimes 1$ in $A(m|n)\otimes A$, then one can interpret this action as the right-hand side matrix multiplication
$C\to Cg$, 
where $C$ and $g$ are given by block matrices
\[C=\left(\begin{array}{cc}
C_{11} & C_{12} \\
C_{21} & C_{22}
\end{array}\right), g=\left(\begin{array}{cc}
A_{11} & A_{12} \\
A_{21} & A_{22}
\end{array}\right).\]

The element $e_{kl}$ from $Lie(G)\subseteq Dist(G)$ acts on $c_{ij}$ as $e_{kl}\cdot c_{ij}=\delta_{jl}c_{ik}$.
Since the last expression $\delta_{jl}c_{ik}=(c_{ij}) _{lk}D$, the action of $e_{kl}$ on $A(m|n)$ is identical to the action of $_{lk}D$ (not $_{kl}D$) on $A(m|n)$.
Analogously, the divided powers $e_{ij}^{(r)}$ for even $e_{ij}$ and $r>0$ correspond to 
$_{ij}^{(r)}D = \frac{_{ij}^rD}{r!}$ and $\binom{e_{ii}}{r}$ corresponds to $\binom{_{ii}D}{r} = \frac{_{ii}D(_{ii}D+1) \ldots (_{ii}D+r-1)}{r!}$ for $1\leq i\leq m+n$ and $r>0$.
For more details please consult \cite{zs}.

Moreover, both actions extends uniquely to $K[G]$ and we will identify the action of $e_{kl}$ on $K[G]$ with the action of the superderivation $_{lk}D$ extended to $K[G]$ using 
the quotient rule
\[(\frac{a}{b})_{ij}D=\frac{(a)_{ij}Db-a(b)_{ij}D}{b^2}\] 
valid for $a,b\in A(m|n)$ and $b$ even.

Next, we would like to recall the process of modular reduction of $\mathbb{C}[G]$ from the field $\mathbb{C}$ of complex numbers to a field $K$ of characteristic $p>2$. 
The Kostant $\mathbb{Z}$-form of $Dist(G)$ corresponds to the $\mathbb{Z}$-subalgebra generated by the expressions $_{ij}D, _{ij}^{(r)}D$ and $\binom{_{ii}D}{r}$ as before.
Consider the $\mathbb{Z}$-span $\mathbb{C}[G]_{\mathbb Z}$ of elements $\frac{C_{I|K}}{(DD_{22})^s}$ for all multi-indices $I,J$ and integers $s\geq 0$. It is clear that 
$\mathbb{C}[G]_{\mathbb Z}$ 
is stabilized by the action of all $_{ij}D, _{ij}^{(r)}D$ and $\binom{_{ii}D}{r}$ and is therefore a $\mathbb{Z}$-form of $\mathbb{C}[G]$. 
By extending the scalars to the field $K$  we obtain the modular reduction $K[G]=\mathbb{C}[G]_{\mathbb{Z}}\otimes_{\mathbb{Z}} K$ of $\mathbb{C}[G]$.

Our main aim is to understand the structure of the induced $G$-module $H^0_G(\lambda)$. In the next subsection we will see that $H^0_G(\lambda)$ is a $Dist(G)$-supersubmodule of $K[G]$
and therefore its $G$-module structure can be described using superderivations $_{ij}D$ (together with various divided powers of these superderivations if the characteristic $p$ of $K$ 
is positive).

Working with even superderivations is somehow easier since they correspond to the action of even subgroup $G_{ev}$ of $G$. 
We will investigate the action of odd superderivations $_{ij}D$ on generating elements of $H^0_{G_{ev}}(\la)$ which are 
given by bideterminants and described in the following subsection. 
It is this crucial step where we need to apply the determinantal identities mentioned in the introduction of the paper.  The corresponding 
formulas and other basic properties are derived in Section 2. The main result of Section 2 is Theorem \ref{gen}.

\subsection{Induced module $H^0_G(\la)$}

Let $K_{\la}$ be the one-dimensional (even) $B$-supermodule corresponding to the weight $\la$. 
Following \cite{z}, we denote by $H^0_G(\lambda)$ the $G$-supermodule
$H^0(G/B,K_{\la})$ which is isomorphic to the induced supermodule $ind^G_P H^0(P/B, K_{\la})$. 
Analogously, denote by $H^0_{G_{ev}}(\lambda)$ the induced $G_{ev}$-module $H^0(G_{ev}/B_{ev},K_{\la})$ 
corresponding to the weight $\lambda$.
The induced $G_{ev }$-module $H^0_{G_{ev}}(\lambda)$ is isomorphic to $H^0_{GL(m)}(\lambda^+)\times H^0_{GL(n)}(\lambda^-)$, where 
$H^0_{GL(m)}(\lambda^+)$ is the classical $GL(m)$-module induced from the lower triangular Borel sugroup of $GL(m)$  and 
$H^0_{GL(n)}(\lambda^-)$ is the classical $GL(n)$-module induced from the lower triangular Borel sugroup of $GL(n)$.

Another description of $H^0_G(\lambda)$ can be obtained using Weyl modules. Following Section 5 of \cite{z}, 
the universal highest weight supermodule (the Weyl supermodule) $V_G(\la)$ can be defined as $V_G(\la)= \Dist(G)\otimes_{\Dist(P^{opp})} V_{G_{ev}}(\la)$, 
where $P^{opp}$ is the parabolic subgroup that is the transpose of $P$ and $V_{G_{ev}}(\lambda)$ is the clasical Weyl module regarded as a $P^{opp}$-supermodule 
via the epimorphism $P^{opp}\to G_{ev}$. By Proposition 5.8 of \cite{z}, $V_G(\la)\cong H^0_G(\la)^{\langle \tau \rangle}$, where the supertransposition $\tau$ provides a contravariant duality.

Next, we will describe induced modules $H^0_{\GL(m)}(\la^+)$, $H^0_{\GL(n)}(\la^-)$ and $H^0_{G_{ev}}(\la)$. 

Assume that $\la^+=(\la_1^+, \ldots, \la_m^+)$ is a dominant integral weight of $GL(m)$,  $\la^-=(\la_1^-, \ldots, \la_n^-)$ is a dominant integral weight of $GL(n)$ and denote
$r^+=\sum_{i=1}^m \la^+_i$ and $r^-=\sum_{j=1}^n \la^+_j$. Then $\la=(\la^+|\la^-)$ is a dominant integral weight of $GL(m|n)$.  

Assume first that $\lambda$ is a polynomial weight of $G$, 
that is $\lambda^+_m\geq 0$ and $%
\lambda^-_n\geq 0$. Define the tableau $T_{\lambda}^+$ of the shape $%
\lambda^+$ as $T_{\lambda}^+(i,j)=i$ for $i=1, \ldots, m$ and $j=1,\ldots,
\lambda_i$ and the tableau $T_{\lambda}^-$ of the shape $\lambda^-$ as $%
T_{\lambda}^-(i,j)=m+i$ for $i=1, \ldots, n$ and $j=1,\ldots, \lambda_{m+i}$. For more about tableaux and bideterminants, see \cite{m}. 

Fix a basic tableau $T^+$ of shape $\lambda^+$ that lists entries from the set $\{1, \ldots r^+\}$ in increasing order 
from left to right of each row starting from the top and proceeding to the bottom row. 
Then multi-indices $I$ of length $r^+$ are in bijective correspondence to tableaux of shape $\lambda^+$. Denote by $T^+(I)$ the tableau corresponding to $I$.
Then $T^+_{\lambda} = T^+(1^{\lambda^+_1} \ldots m^{\lambda^+_m})$.
Analogously define the basic tableau $T^-$ and by $T^-(J)$ denote the tableau of shape $\lambda^-$ corresponding to the multi-index $J$ of length $r^-$. 
Then $T_{\lambda}^-=T^-((m+1)^{\lambda^-_1} \dots (m+n)^{\lambda^-_n})$.

For a multi-index $I$ of length $r^+$ with (possibly repeated) entries from the set $1, \ldots, m$ denote by $B_+(I)=(T_{\lambda}^+:T^+(I))$ the bideterminant corresponding 
to the tableau $T^+(I)$ of shape $\lambda^+$. Then the induced $\GL(m)$-module of the highest weight $\lambda^+$ is the $K$-span of bideterminants $B_+(I)$
and it has a basis consisting of bideterminants $B_+(I)$ corresponding to standard tableau $T^+(I)$.
Analogously, for a multi-index $J$ of length $r^-$ with (possibly repeated) entries from the set $m+1, \ldots, m+n$ denote by $B_-(J)=(T_{\lambda}^-:T^-(J))$ the bideterminant corresponding to the tableau $T^-(J)$ of shape $\lambda^-$. Then the induced $\GL(n)$-module of the highest weight $\lambda^-$ is the $K$-span of bideterminants $B_-(J)$ 
and it has a basis consisting of bideterminants $B_-(J)$ corresponding to standard tableux $T^-(J)$. 

Consequently, the induced module $H^0_{G_{ev}}(\lambda)$ is spanned by the products $B_+(I)B_-(J)$, where $T^+(I)$ is of shape $\lambda^+$ and $T^-(J)$ is of shape $\lambda^-$. 
It has a basis consisting of such products $B_+(I)B_-(J)$, where both $T^+(I)$ and $T^-(J)$ are standard tableaux.

Now consider the case when $\la$ is not polynomial. If $\la^+_m<0$, 
then the induced $\GL(m)$-module $H^0_{\GL(m)}(\la^+)$ is isomorphic to
$H^0_{\GL(m)}(\la^{++})\otimes (D^{\la^+_m})$, where 
the weight $\la^{++}=\la^+ -\la^+_m(1,1, \ldots 1)$ is polynomial and $(D^{\la^+_m})$ is 
a one-dimensional $\GL(m)$-representation generated by $D^{\la^+_m}$.
Therefore the module $H^0_{\GL(m)}(\la^+)$ has a basis consisting of
products of bideterminants $B_+(I)$ for $I$ standard of shape $\lambda^{++}$ multiplied by $D^{\la^+_m}$.

In the supercase, there is a group-like element \[Ber(C)=\de(C_{11}-C_{12}C_{22}^{-1}C_{21})\de(C_{22})^{-1}\] 
which generates an irreducible one-dimensional $G$-module $Ber$ of the weight 
$\beta=(1,1,\ldots, 1|-1,-1,\ldots, -1)$.
Since $H^0_G(\la)\cong H^0_G(\la-\la^-_n\beta) \otimes Ber^{\la^-_n}$, in what follows we can assume that 
$\la^-$ is a polynomial weight of $\GL(n)$, that is $\la^-_n\geq 0$, since we can reduce the general case to this one by 
tensoring with $Ber^{\la^-_n}$.
 
Assuming $\la^-_n\geq 0$, the module $H^0_{G_{ev}}(\lambda)$ has a basis that is a product of $D^{\la^+_m}$, 
bideterminants $B_+(I)$ for $I$ standard of shape $\lambda^{++}$ and
bideterminants $B_-(J)$ for $J$ standard of shape $\lambda^-$.
Denote $v_{++}=(T_{\lambda^{++}}:T_{\lambda^{++}})$, $v_+=v_{++}D^{\la^+_m}$ and $v_-=(T_{\lambda}^-:T_{\la}^-)$. 
Then the element $v_+v_-$ is the highest weight vector in $H^0_{G_{ev}}(\lambda)\cong H^0_{GL(m)}(\lambda^+)\times H^0_{GL(n)}(\lambda^-)$. 

The following segments explicitly identify the induced module $H^0_G(\la)$ as a supersubmodule of $K[G]$. Based on our previous discussion, its $G$-module structure can be then described (later) by computing the action of superderivations on its elements.

For a superspace $V$ denote by $S(V)$ the supersymmetric superalgebra generated by the superspace $V$. It has a natural grading and its $t$-homogeneous component will be denoted by
$S^t(V)$. In particular, $S(C_{12})$ is the supersymmetric superalgebra $S(\bigoplus\limits_{\substack{1\leq i\leq m \\ m+1\leq j \leq m+n}} Kc_{ij})$, 
which is a subsuperalgebra of $A(m|n)$. 

According to Remark 5.1 of \cite{z}, there is an isomorphism of affine superschemes $G\to P\times U$ 
defined by 
\begin{equation*}
\begin{pmatrix}
C_{11} & C_{12} \\ 
C_{21} & C_{22}
\end{pmatrix}
\mapsto 
\Big( \begin{pmatrix}
C_{11} & 0 \\ 
C_{21} & C_{22}-C_{21}C_{11}^{-1}C_{12}
\end{pmatrix},
\begin{pmatrix}
E_m & C_{11}^{-1}C_{12} \\ 
0 & E_n
\end{pmatrix} \Big).
\end{equation*}

The inverse morphism is given by the multiplication map.
The dual morphism $K[P]\otimes K[U] \to K[G]$ is a superalgebra isomorphism. 

Consider the $G_{ev}$-action on $K[P]$ via right multiplication and observe that $K[U]\simeq S(C_{12})$. 
The action of $G_{ev}$ on $P\times U$ induces the $G_{ev}$-module structure on $K[P]\otimes S(C_{12})$ given by the rule 
$(p, u)\cdot g=(pg, g^{-1}ug)$ for $p\in P$, $u\in U$ and $g\in G_{ev}$.
Then the morphism $G\to P\times U$ is $G_{ev}$-equivariant and so is $K[P]\otimes S(C_{12})\to K[G]$. 
In particular, the image of $S(C_{12})$ is a $G_{ev}$-submodule of $K[G]$.

The $G$-supermodule $H^0_G(\lambda)$ is described explicitly using the
isomorphism $\tilde{\phi} :H^0_{G_{ev}}(\lambda)\otimes S(C_{12})\to H^0_G(\lambda)$ of superspaces
defined in Lemma 5.1 of \cite{z}. This map is a restriction of the multiplicative morphism $\phi:K[G]\to K[G]$ given on generators as
follows: 
\begin{equation*}
C_{11}\mapsto C_{11}, C_{21}\mapsto C_{21}, C_{12}\mapsto
C_{11}^{-1}C_{12}, C_{22}\mapsto C_{22}-C_{21}C_{11}^{-1}C_{12}.
\end{equation*}

The action of $G$ on $K[G]$ restricts to the action of $G_{ev}$ on $H^0_{G_{ev}}(\lambda)\otimes S(C_{12})$, considered as a subspace of $K[G]$. However, we are going to consider a 
different $G_{ev}$-action on $H^0_{G_{ev}}(\lambda)\otimes S(C_{12})$, defined below, such that the above map $\tilde{\phi}$ is a morphism of $G_{ev}$-modules
and the $G_{ev}$-structure on $\tilde{\phi}(H^0_{G_{ev}}(\lambda)\otimes S(C_{12}))$ is induced from the structure of the $G_{ev}$-module $H^0_G(\lambda)|_{G_{ev}}$.

The group $G_{ev}$ acts on $G$ by multiplication on the right as follows. For each commutative superalgebra $A$ and elements
$g=\left(\begin{array}{cc} C_{11} & C_{12} \\ C_{21} & C_{22} \end{array}\right)\in G(A)$ and 
$h=\left(\begin{array}{cc} B_{11} & 0 \\0 & B_{22} \end{array}\right)\in G_{ev}(A)$, the formula
$g.h= \left(\begin{array}{cc}
C_{11}B_{11} & C_{12}B_{22} \\
C_{21}B_{11} & C_{22}B_{22}
\end{array}\right)$ defines the action of $G_{ev}$ on the generators of $A(m|n)$. This action extends linearly and multiplicatively to all elements of $A(m|n)$ 
and defines the $G_{ev}$-action on $GL(m|n)$.

The action of $h$ on images under $\phi$ is given by
\[\phi(C_{11}).h = C_{11}B_{11}=\phi(C_{11})B_{11},\] 
\[\phi(C_{21}).h = C_{21}B_{11}=\phi(C_{21})B_{11},\]
\[\phi(C_{12}).h = (C_{11}^{-1}C_{12}).h=B_{11}^{-1}C_{11}^{-1}C_{12}B_{22}=B_{11}^{-1}\phi(C_{12})B_{22}\]
and 
\[\phi(C_{22}).h = (C_{22}-C_{21}C_{11}^{-1}C_{12}).h=(C_{22}-C_{21}C_{11}^{-1}C_{12})B_{22}=\phi(C_{22})B_{22}.\] 
Therefore the restriction of $\tilde{\phi}$ to $H^0_{G_{ev}}(\lambda)\otimes 1$ is a morphism of $G_{ev}$-modules.
Let us define the $G_{ev}$-structure on the superspace $C_{12}$ via 
$C_{12}.h = B_{11}^{-1}C_{12}B_{22}$ and denote this $G_{ev}$-module by $Y$.
Then the map $\tilde{\phi}:H^0_{G_{ev}}\otimes S(Y) \to H^0_G(\lambda)$ is a morphism of $G_{ev}$-modules and the $G_{ev}$-action on the image of $\tilde{\phi}$ is induced from 
$H_G(\lambda)|_{G_{ev}}$. In particular, the maps $\phi$ and $\tilde{\phi}$ preserve the weights. 

The supersymmetric superalgebra $S(Y)$, considered as a $G_{ev}$-module, is isomorphic to the exterior algebra $\Lambda(Y)$. 
The $G_{ev}$-structure of $Y$ can be easily described as follows. Denote by $V_m$ the natural (defining) representation of $\GL(m)$-module on $m\times 1$ column vectors
with the highest vector $v_1$ of the weight $\epsilon_1=(1,0, \ldots, 0)$. The module
$V_m$ is spanned by vectors $v_i=e_{i1}v_1$ for $i=1, \ldots, m$, where $e_{i1}\in \Dist(\GL(m))$ is the appropriate matrix unit. 
The dual $\GL(m)$-module $(V_m)^*$ has the highest vector $(v_m)^*$ of the weight $-\epsilon_m=(0, \ldots, 0, -1)$. 
Denote by $V_n$ the natural representation of $\GL(n)$-module on $n\times 1$ column vectors.
Then the $G_{ev}$-module $(V_m)^*\otimes V_n$, where the action of $\GL(m)$ on $V_n$ and the action of $\GL(n)$ 
on $(V_m)^*$ are trivial, is isomorphic to $Y$. 

We have described earlier the structure of $H^0_{G_{ev}}(\la)$ using products of bideterminants. When we are working in $H^0_G(\la)$, we need to replace $H^0_{G_{ev}}(\la)$ with 
its image under the map $\phi$. From now on, we will consider $H^0_{G_{ev}}(\la)$ embedded inside $H^0_G(\la)$ using the map $\phi$. Because of this we need to adjust the notation for bideterminants in order to accomodate the effect of the map $\phi$.

Recall that the matrix $A=(A_{ij})$ is the adjoint of the matrix $C_{11}$. Then $C_{11}^{-1}=\frac{1}{D}A$
and 
\begin{equation*}
y_{ij}=\phi(c_{ij})=\frac{A_{i1}c_{1j}+A_{i2}c_{2j}+\ldots +A_{im}c_{mj}}{D}
\end{equation*}
for $1\leq i\leq m$ and $m+1\leq j \leq m+n$. Moreover, for $m+1\leq k,l
\leq m+n$ we have 
\begin{equation*}
\phi(c_{kl})=c_{kl}-c_{k1}y_{1l} - \ldots -c_{km}y_{ml}.
\end{equation*}

Let $M=(m_{ij})$ be a matrix of size $s\times s$. Let $\{i_{1},\ldots
,i_{t}\}$ and $\{j_{1},\ldots ,j_{t}\}$ be sequences of elements from $%
\{1,\ldots ,s\}$. Denote by $M(i_{1},\ldots ,i_{t}|j_{1},\ldots j_{t})$ the
determinant of the matrix of size $t\times t$ such that its entry in the $a$%
-th row and the $b$-th column is $m_{i_{a},j_{b}}$. If the entries in $%
\{i_{1},\ldots ,i_{t}\}$ and $\{j_{1},\ldots ,j_{t}\}$ are pairwise
different, then $M(i_{1},\ldots ,i_{t}|j_{1},\ldots j_{t})$ is the $t$-th
minor of $M$ corresponding to the rows $\{i_{1},\ldots ,i_{t}\}$ and columns 
$\{j_{1},\ldots ,j_{t}\}$ of matrix $M$. We will use this notation for the
matrix $C$ and further denote $C(1,2,\ldots ,t|j_{1},\ldots ,j_{t})$ by $%
C(j_{1},\ldots ,j_{t})$.

If $1\leq i_1, \ldots, i_s \leq m$, then denote by $D^+(i_1, \ldots, i_s)$
the determinant 
\begin{equation*}
\begin{array}{|ccc|}
c_{1,i_1} & \ldots & c_{1,i_s} \\ 
c_{2,i_1} & \ldots & c_{2,i_s} \\ 
\ldots & \ldots & \ldots \\ 
c_{s,i_1} & \ldots & c_{s,i_s}%
\end{array}%
.
\end{equation*}
Clearly, if some of the numbers $i_1, \ldots, i_s$ coincide, then $%
D^+(i_1,\ldots, i_s)=0$.

First, we want to express bideterminants corresponding to a tableau $T(I)$ of shape $\lambda^+$ using the above determinants $D^+$.
Assume that $\la^+$ is a polynomial weight and a tableau $T(I)$ of shape $\lambda^+$ is such that its entry in
the $a$-th row and $b$-column is $i_{ab}\in\{1, \ldots, m\}$. The bideterminant $B^+(I)$ is a
product of determinants $D^+(i_{1b},\ldots, i_{m,b})$ for $b=1, \ldots
\lambda^+_m$, $D^+(i_{1b},\ldots, i_{m-1,b})$ for $b=\lambda^+_m+1, \ldots
\lambda^+_{m-1}$, \ldots, and $D^+(i_{1b})$ for $b=\lambda^+_2+1,\ldots,
\lambda^+_1$. If we denote the length of the $b$-th column of $T(I)$ by $%
\ell(b)$, we can write 
\begin{equation*}
B^+(I)=\prod_{b=1}^{\lambda^+_1}D^+(i_{1b},\ldots, i_{\ell(b),b}).
\end{equation*}
It is useful to observe that the weight of the bideterminant $B^+(I)$ is $cont(I)$ (see also pp.42-43 of \cite{m}).
Let us note that there is no essential distinction between previously defined $B_+(I)$ and $B^+(I)$. This is because $\phi(c_{i,j})=c_{i,j}$ for 
$1\leq i\leq m$ and $1\leq j\leq m$. 

If $m+1\leq j_1, \ldots, j_s \leq m+n$, then denote by $D^-(j_1, \ldots, j_s)
$ the determinant 
\begin{equation*}
\begin{array}{|ccc|}
\phi(c_{m+1,j_1}) & \ldots & \phi(c_{m+1,j_s}) \\ 
\phi(c_{m+2,j_1}) & \ldots & \phi(c_{m+2,j_s}) \\ 
\ldots & \ldots & \ldots \\ 
\phi(c_{m+s,j_1}) & \ldots & \phi(c_{m+s,j_s})%
\end{array}%
.
\end{equation*}
Clearly, if some of the numbers $j_1, \ldots, j_s$ coincide, then $%
D^-(j_1,\ldots, j_s)=0$.

Assume that a tableau $T(J)$ of shape $\lambda^-$ is such that its entry in
the $a$-th row and $b$-column is $j_{ab}\in\{m+1, \ldots, m+n\}$. The bideterminant $B^-(J)$ is a
product of determinants $D^-(j_{1b},\ldots, j_{n,b})$ for $b=1, \ldots
\lambda^-_n$, $D^-(j_{1b},\ldots, j_{n-1,b})$ for $b=\lambda^-_n+1, \ldots
\lambda^-_{n-1}$, \ldots, and $D^-(j_{1b})$ for $b=\lambda^-_2+1,\ldots,
\lambda^-_1$. If we denote the length of the $b$-th column of $T(J)$ by $%
\ell(b)$, we can write 
\begin{equation*}
B^-(J)=\prod_{b=1}^{\lambda^-_1}D^-(j_{1b},\ldots, j_{\ell(b),b}).
\end{equation*}
Similar to the previous case, the weight of $B^-(J)$ is $cont(J)$.
The expression $B^-(J)$ will be also called bideterminant corresponding to $J$ and it corresponds to $\phi(B_-(J))$, the image of previously defined 
bideterminant $B_-(J)$ under the map $\phi$. 

Since the highest vector $v_+$ of $H^0_{\GL(m)}(\la^+)$ and the highest vector $v_-$ of $H^0_{\GL(n)}(\la^-)$ were described earlier,
their images $v^+$ and $v^-$ respectively under the map $\phi$ are identified as the following elements of $H^0_{G_{ev}}(\la)$: 
\[v^+=\prod_{a=1}^m D^+(1,\ldots, a)^{\la^+_a-\la^+_{a+1}}, \qquad
v^-=\prod_{b=1}^n D^-(m+1, \ldots, m+b)^{\la^-_b - \la^-_{b+1}},\]
where $\la^+_{m+1}=0=\la^-_{n+1}$. Therefore the product $v=v^+v^-$ is the highest vector of $H^0_G(\lambda)$.

\section{Basic formulas}

Some of the formulas derived in this section were mentioned in our previous paper \cite{fm2}. The treatment given here extends these results
and is more comprehensive.

\subsection{Action of superderivations $_{kl}D$ on $\phi(c_{ij})$ for $m+1\leq j \leq m+n$}

We start by computing the action of superderivations $_{kl}D$ on elements $y_{ij}$.
The following statement is Lemma 2.1 of \cite{fm2}. We include its proof here for the convenience of the reader.

\begin{lm}\label{1} If $1\leq i,k \leq m$ and $m+1\leq j,l \leq m+n$, then 
\begin{equation*}
(y_{ij})_{kl}D=y_{il}y_{kj}.
\end{equation*}
\end{lm}
\begin{proof}
First we show that $(D)_{kl}D=Dy_{kl}$. Write $D=A_{k1}c_{1k}+\ldots +A_{km}c_{mk}$. Since $(-1)^{k+j}A_{kj}$
is the determinant of the matrix obtained by removing the $j$-th row and $k$th column from $C_{11}$, it follows that 
$A_{kj}$ is a sum of monomials in the variables $c_{rs}$ for $r\neq j$ and $s\neq k$. Therefore by the superderivation property of $_{kl}D$
we infer that $(A_{kj})_{kl}D=0$. Since $(c_{ik})_{kl}D=c_{il}$ we conclude that $(D)_{kl}D=A_{k1}c_{1l}+\ldots +A_{km}c_{ml}=Dy_{kl}$.

Assume that $i=k$. Since $Dy_{kj}=A_{k1}c_{1j}+\ldots +A_{km}c_{mj}$, $(A_{ka})_{kl}D=0$ and $(c_{aj})_{kl}D=0$ for every $a=1, \ldots m$, 
using the superderivation property of $_{kl}D$ we derive $(Dy_{kj})_{kl}D=0$. Then $(Dy_{kj})_{kl}D=-Dy_{kl}y_{kj}+D(y_{kj})_{kl}D$ 
implies that $(y_{kj})_{kl}D=y_{kl}y_{kj}$.

Therefore we can assume $m\geq 2$. For $a\neq c$ and $b\neq d$, denote by $%
M(ab|cd)$ the $(m-2)\times (m-2)$ minor of the matrix $C_{11}$ obtained by
deleting $a$-th and $b$-th row and $c$-th and $d$-th columns. If $a=b$ or $%
c=d$, then set $M(ab,cd)=0$.

Now assume that $i<k$. Expanding the determinant representing $A_{ib}$ by
the column containing entries from the $k$-th column of the matrix $C_{11}$
we obtain 
\begin{equation*}
(Dy_{ij})_{kl}D=(A_{i1}c_{1j}+\ldots + A_{im}c_{mj})_{kl}D
\end{equation*}
\begin{equation*}
=\sum_{b=1}^m \sum_{a=1}^{b-1}  (-1)^{a+b+k+i} M(ab, ki) c_{al}c_{bj}
-\sum_{b=1}^m \sum_{a=b+1}^{m} (-1)^{a+b+k+i} M(ab, ki) c_{al}c_{bj}.
\end{equation*}

On the other hand,

\begin{equation*}
(Dy_{il})(Dy_{kj})=(A_{i1}c_{1l}+\ldots + A_{im}c_{ml})(A_{k1}c_{1j}+\ldots
+ A_{km}c_{mj})= \sum_{b=1}^m \sum_{a=1}^m A_{ia}A_{kb}c_{al}c_{bj}
\end{equation*}
and 
\begin{equation*}
(Dy_{kl})(Dy_{ij})=(A_{k1}c_{1l}+\ldots + A_{km}c_{ml})(A_{i1}c_{1j}+\ldots
+ A_{im}c_{mj})= \sum_{b=1}^m \sum_{a=1}^m A_{ka}A_{ib}c_{al}c_{bj}
\end{equation*}
implies that 
\begin{equation*}
(Dy_{il})(Dy_{kj})-(Dy_{kl})(Dy_{ij})=\sum_{b=1}^m \sum_{a=1}^m
(A_{ia}A_{kb}-A_{ka}A_{ib})c_{al}c_{bj}.
\end{equation*}

Using the Jacobi Theorem on minors of the adjoint matrix (see \cite{g},
p.21, \cite{f}, p.57 or Theorem 2.5.2 of \cite{p}) we have that $%
A_{ia}A_{kb}-A_{ka}A_{ib}=(-1)^{a+b+k+i} D M(ab, ki)$ for $a<b$ and $%
A_{ia}A_{kb}-A_{ka}A_{ib}=(-1)^{a+b+k+i+1} D M(ab, ki)$ for $a>b$.

Therefore 
\begin{equation*}
(Dy_{il})(Dy_{kj})-(Dy_{kl})(Dy_{ij})=D(Dy_{ij})_{kl}D=-D^2y_{kl}y_{ij}+D^2(y_{ij})_{kl}D
\end{equation*}
and this implies $(y_{ij})_{kl}D=y_{il}y_{kj}$.

The remaining case $i>k$ can be handled analogously.

\end{proof} 

\begin{lm}\label{1'} If $1\leq i,k \leq m$ and $m+1\leq j,l \leq m+n$, then 
\begin{equation*}
(y_{ij})_{lk}D=\delta_{ik}\delta_{jl}.
\end{equation*}
\end{lm}
\begin{proof}
Since $y_{ij}=\frac{A_{i1}c_{1j}+A_{i2}c_{2j}+\ldots +A_{im}c_{mj}}{D}$, it is clear that $(y_{ij})_{lk}D=0$ unless
$l=j$. If $l=j$, then $(y_{ij})_{jk}D=\frac{A_{i1}c_{1k}+A_{i2}c_{2k}+\ldots +A_{im}c_{mk}}{D}=\delta_{ik}$.
\end{proof}

\begin{lm}\label{1''} If $1\leq i \leq m$ and $m+1\leq j,k,l \leq m+n$, then 
\begin{equation*}
(y_{ij})_{kl}D=\delta_{jk}y_{il}.
\end{equation*}
\end{lm}
\begin{proof}
Since $(Dy_{ij})_{kl}D=(A_{i1}c_{1j}+A_{i2}c_{2j}+\ldots +A_{im}c_{mj})_{kl}D=\delta_{jk}(A_{i1}c_{1l}+A_{i2}c_{2l}+\ldots +A_{im}c_{ml})=
\delta_{jk}Dy_{il}$ and $(D)_{kl}D=0$, the claim follows. 
\end{proof}

\begin{lm}\label{1'''} 
If $1\leq i ,k, l\leq m$ and $m+1\leq j\leq m+n$, then 
\begin{equation*}
(y_{ij})_{kl}D=-\delta_{li}y_{kj}.
\end{equation*}
\end{lm}
\begin{proof}
First we show that $(D)_{kl}D=\delta_{kl}D$. Write $D=A_{k1}c_{1k}+\ldots +A_{km}c_{mk}$. Since $(-1)^{k+j}A_{kj}$
is the determinant of the matrix obtained by removing the $j$-th row and $k$th column from $C_{11}$, it follows that 
$A_{kj}$ is a sum of monomials in the variables $c_{rs}$ for $r\neq j$ and $s\neq k$. Therefore by the superderivation property of $_{kl}D$
we infer that $(A_{kj})_{kl}D=0$. Since $(c_{ik})_{kl}D=c_{il}$, we conclude that $(D)_{kl}D=A_{k1}c_{1l}+\ldots +A_{km}c_{ml}=\delta_{kl}D$.

Assume that $i=k$. Since $Dy_{kj}=A_{k1}c_{1j}+\ldots +A_{km}c_{mj}$, $(A_{ka})_{kl}D=0$ and $(c_{aj})_{kl}D=0$
for every $a=1, \ldots m$, using the superderivation property of $_{kl}D$ we derive $(Dy_{kj})_{kl}D=0$. 
Then $(Dy_{kj})_{kl}D=(D)_{kl}Dy_{kj}+D(y_{kj})_{kl}D$ implies that $(y_{kj})_{kl}D=-\delta_{kl}y_{kj}$.

Assume that $k\neq i$. Since $(-1)^{i+a}A_{ia}$ is the determinant of the matrix obtained by removing the $a$-th row and $i$-th column from $C_{11}$,
the application of the superderivation $_{kl}D$ on $(-1)^{i+a}A_{ia}$ has the same effect as replacing the $k$-th column of $C_{11}$ by the $l$-th column of $C_{11}$ prior to 
deleting $a$-th row and $i$-th column and computing the determinant. 
If $l\neq i$, then this determinant has two identical columns, and therefore, $(A_{ia})_{kl}D=0$. If $l=i$, then we obtain 
$(A_{ia})_{kl}D=-A_{ka}$. 
Thus $(Dy_{ij})_{kl}D=(A_{i1}c_{1j}+A_{i2}c_{2j}+\ldots +A_{im}c_{mj})_{kl}D=-\delta_{li}(A_{k1}c_{1j}+A_{k2}c_{2j}+\ldots +A_{ik}c_{mj})=
-\delta_{li}Dy_{kj}$ implies the last claim. 
\end{proof}

Now we compute the action of superderivations $_{kl}D$ on elements $\phi(c_{ij})$ when $m+1\leq i,j \leq m+n$.
The next statement is Lemma 2.2 of \cite{fm2}.

\begin{lm}
\label{2} If $1\leq k \leq m$ and $m+1\leq i,j,l \leq m+n$, then 
\begin{equation*}
(\phi(c_{ij}))_{kl}D=\phi(c_{il})y_{kj}.
\end{equation*}
\end{lm}
\begin{proof}
Since $\phi(c_{ij})=c_{ij}-c_{i1}y_{1j} - \ldots -c_{im}y_{mj}$, using Lemma %
\ref{1} we compute 
\begin{equation*}
(\phi(c_{ij}))_{kl}D=-c_{i1}y_{1,l}y_{kj} - c_{i2}y_{2l}y_{kj} -\ldots
+c_{il}y_{kj} - c_{ik}y_{kl}y_{kj} -\ldots -c_{im}y_{ml}y_{kj}=
\phi(c_{il})y_{kj}.
\end{equation*}
\end{proof}

\begin{lm}
\label{2''} 
If $1\leq k \leq m$ and $m+1\leq i,j,l \leq m+n$, then 
\begin{equation*}
(\phi(c_{ij}))_{lk}D=0.
\end{equation*}
\end{lm}
\begin{proof}
Using Lemma \ref{1'} we compute
$(\phi(c_{ij}))_{lk}D=(c_{ij}-c_{i1}y_{1j}-\ldots -c_{im}y_{mj})_{lk}D=\delta_{jl}c_{ik}-\delta_{jl}c_{ik}=0$.
\end{proof}

Although the map $\phi: K[G]\to K[G]$ is not a morphism of $G_{ev}$-modules, the following statement gives a useful property for a restriction of $\phi$.  

\begin{lm}
\label{2'} 
If $m+1\leq i,j,k,l \leq m+n$, then 
\begin{equation*}
(\phi(c_{ij}))_{kl}D=\delta_{jk}\phi(c_{il})=\phi((c_{ij})_{kl}D).
\end{equation*}
\end{lm}
\begin{proof}
Since $\phi(c_{ij})=c_{ij}-c_{i1}y_{1j}-\ldots -c_{im}y_{mj}$, we see immediately that 
$(\phi(c_{ij}))_{kl}D=\delta_{jk} \phi(c_{il})$.
On the other hand, $(c_{ij})_{kl}D=\delta_{jk} c_{il}$ and the claim follows. 
\end{proof}

\begin{lm}
\label{2'''} 
If $1\leq k,l\leq m$ and $m+1\leq i,j\leq m+n$, then $(\phi(c_{ij}))_{kl}D=0$.
\end{lm}
\begin{proof}
We have $(y_{lj})_{kl}D=-y_{kj}$ and $(y_{aj})_{kl}D=0$ for $a\neq l$ by Lemma \ref{1'''}.
Since  $\phi(c_{ij})=c_{ij}-c_{i1}y_{1j}-\ldots -c_{im}y_{mj}$, we infer that $(\phi(c_{ij}))_{kl}D=-c_{il}y_{kj}+c_{il}y_{kj}=0$.
\end{proof}

\subsection{Action of superderivations $_{kl}D$ on determinants $D^+(i_1, \ldots, i_s)$}

We need the following property of certain determinants.

\begin{lm}
\label{new} Let $0\leq j <m$, $1\leq k_1, \ldots k_j \leq m$ and $m+1\leq l\leq m+n$. Then 
\begin{equation*}
C(k_1, \ldots, k_j, l)
=\sum_{a=1}^m C(k_1,\ldots, k_j, a)y_{a,l}
\end{equation*}
\end{lm}
\begin{proof}
Without a loss of generality we can assume $k_1<\ldots < k_i< \ldots < k_j$.

Use the Laplace expansion along the $(j+1)$-th column for the minor $C(1, \ldots, j+1|k_1, \ldots, k_j,l)$
of the matrix $C_{11}$  to get 
\begin{equation*}
C(k_1, \ldots, k_j, l)= \sum_{s=1}^{j+1} (-1)^{s+j+1} 
C(1, \ldots, \hat{s}, \ldots, j+1|k_1, \ldots, k_j) c_{sl}.
\end{equation*}

The right-hand side of the last formula can be expressed as
\begin{equation*}
\sum_{s=1}^m \sum_{a=1}^{m} C(k_1,\ldots, k_j, a) A_{a,s} \frac{c_{sl}}{D}.
\end{equation*}
provided we show that 
\begin{equation*}
\sum_{a=1}^{m} C(k_1,\ldots, k_j, a) A_{a,s}= (-1)^{s+j+1}C(1,\ldots, \hat{s}, \ldots, j+1|k_1, \ldots, k_j)D
\end{equation*}
if $s\leq j+1$, and that $\sum_{a=1}^{m} C(k_1,\ldots, k_j, a) A_{a,s}=0$ otherwise.

These equalities are obtained by using the method of extension (Muir's law
of extensible minors) of determinantal identities - see Sections 7 an 8 of 
\cite{bs}. Assume $s\leq j+1$. Denote by 
$l_1< \ldots< l_{m-j}$ the listing of elements of $\{1, \ldots, m\}\setminus
\{k_1, \ldots, k_j\}$. 

The Laplace expansion of $C(s,j+2, \ldots,
m|l_1,\ldots,l_{m-j})$ along the first row gives 
\begin{equation*}
C(\emptyset|\emptyset)C(s,j+2, \ldots, m|l_1,\ldots,l_{m-j})=
\end{equation*}
\begin{equation*}
\sum_{b=1}^{m-j} (-1)^{b+1} C(s|l_b)C(j+2, \ldots, m|l_1, \ldots, \hat{l_b},\ldots, l_{m-j}).
\end{equation*}

By adjoining the symbols $(1, \ldots, \hat{s}, \ldots, j+1|k_1, \ldots, k_j)$ we obtain 
\begin{equation*}
C(1, \ldots, \hat{s}, \ldots, j+1|k_1, \ldots, k_j)
C(1, \ldots, \hat{s}, \ldots, j+1, s,j+2, \ldots, m|k_1, \ldots, k_j,l_1,\ldots,l_{m-j})=
\end{equation*}

\begin{equation*}
\sum_{b=1}^{m-j} (-1)^{b+1} C(1, \ldots, \hat{s}, \ldots, j+1,s|k_1, \ldots, k_j,l_b) \times
\end{equation*}
\begin{equation*}
C(1, \ldots, \hat{s}, \ldots m|k_1, \ldots, k_j,l_1, \ldots, \hat{l_b},\ldots, l_{m-j}).
\end{equation*}

Denote by $\tau$ the permutation $(k_1, \ldots, k_j,l_1,\ldots,l_{m-j})$ of the set $\{1, \ldots, m\}$.
and by $\tau_b$ the permutation $(k_1, \ldots, k_j,l_1,\ldots, \hat{l_b}, \ldots,l_{m-j})$ of the set $\{1, \ldots,\hat{l_b}, \ldots  m\}$.
Then $(-1)^{|\tau_b|}=(-1)^{|\tau|}(-1)^{j+b+l_b}$. 

Since  \[C(1, \ldots, \hat{s}, \ldots, j+1, s,j+2, \ldots, m|k_1, \ldots, k_j,l_1,\ldots,l_{m-j})=(-1)^{s+j+1}(-1)^{|\tau|} D\] and 
\[C(1, \ldots, \hat{s}, \ldots, j+1,s|k_1, \ldots, k_j,l_b)=(-1)^{s+j+1} C(k_1,\ldots, k_j,l_b),\] 
the previous equation implies 
\begin{equation*}
(-1)^{s+j+1}(-1)^{|\tau|}C(1, \ldots, \hat{s}, \ldots, j+1|k_1, \ldots, k_j)D=
\end{equation*}
\begin{equation*}
=\sum_{b=1}^{m-j}(-1)^{s+j+1} (-1)^{b+1}  C(k_1, \ldots, k_j,l_b) 
C(1,\ldots, \hat{s},\ldots, m|k_1, \ldots, k_j,l_1, \ldots, \hat{l_b},\ldots, l_{m-j})
\end{equation*}
\begin{equation*}
=\sum_{b=1}^{m-j} (-1)^{s+j+1}(-1)^{|\tau|}(-1)^{j+1+l_b}  C(k_1, \ldots, k_j,l_b) 
C(1,\ldots, \hat{s},\ldots, m|1, \ldots,\hat{l_b},\ldots, m)
\end{equation*}
\begin{equation*}
=(-1)^{|\tau|}(-1)^{s+j+1}\sum_{b=1}^{m-j} (-1)^{j+1+l_b}  C(k_1, \ldots, k_j,l_b) (-1)^{s+l_b}A_{l_bs}
\end{equation*}
and 
\begin{equation*}
C(1, \ldots, \hat{s}, \ldots, j+1|k_1, \ldots, k_j)D
=(-1)^{j+1+s}\sum_{b=1}^{m-j} C(k_1, \ldots, k_j,l_b) A_{l_bs}.
\end{equation*}

If $s>j+1$, then we can use again the Laplace expansion of $C(s,j+2, \ldots,
m|l_1,\ldots,l_{m-j})$ along the first row. However, in this case 
$C(s,j+2,\ldots, m|l_1,\ldots,l_{m-j})=0$ and 
$\sum_{a=1}^{m} C(k_1, \ldots, k_j,a) A_{as}=\sum_{b=1}^{m-j} C(k_1, \ldots, k_j,l_b) A_{l_bs}=0$.
\end{proof}

The action of $_{kl}D$ on determinants $D^+(i_1, \ldots, i_s)$ is given in the following statements.

\begin{lm}
\label{4} Let $1\leq i_1< \ldots < i_s \leq m$, $1\leq k\leq m$ and $m+1\leq l\leq m+n$. 
If $i_t=k$ for some $1\leq t\leq s$, then 
\begin{equation*}
D^+(i_1, \ldots, i_s) _{kl}D=\sum_{a=1}^m D^+(i_1,\ldots, i_{t-1}, a, i_{t+1}, \ldots, i_s)y_{a,l}.
\end{equation*}
If none of the $i_1, \ldots, i_s$ equals $k$, then $D^+(i_1, \ldots, i_s) _{kl}D=0$.
\end{lm}
\begin{proof}
If $i_t=k$, then $D^+(i_1,\ldots, i_s) _{kl}D=C(i_1, \ldots, i_{t-1}, l, i_{t+1}, \ldots, i_s)$ and the claim follows from Lemma \ref{new}. The other case is obvious.
\end{proof}

\begin{lm} \label{4'}
Assume $1\leq i_1< \ldots < i_s \leq m$. If $1\leq k,l \leq m$ and $k=i_j$ for some $j=1, \ldots, s$, then $D^+(i_1,\ldots, i_s)_{kl}D=D^+(i_1, \ldots, i_{j-1}, l, i_{j+1}, \ldots, i_s)$. If none of the $i_j$ equals $k$, then $D^+(i_1,\ldots, i_s)_{kl}D=0$.
If $m+1\leq k\leq m+n$, then $D^+(i_1,\ldots, i_s)_{kl}D=0$.
\end{lm}
\begin{proof}
The proof is obvious.
\end{proof}

\begin{lm} \label{4''}
Assume $1\leq i_1< \ldots < i_s \leq m$. If $m+1\leq k \leq m+n$ and $1\leq l\leq m+n$, then $D^+(i_1,\ldots, i_s)_{kl}D=0$.
\end{lm}
\begin{proof}
The proof is obvious.
\end{proof}

\subsection{Action of superderivations $_{kl}D$ on determinants $D^-(j_1, \ldots, j_s)$}

The following statement is Lemma 2.3 of \cite{fm2}.

\begin{lm}
\label{3} If $1\leq k \leq m$ and $m+1\leq l,j_1, \ldots, j_s \leq m+n$,
then 
\begin{equation*}
(D^-(j_1, \ldots, j_s))_{kl}D=D^-(l,j_2, \ldots, j_s) y_{kj_1}+D^-(j_1,l,
\ldots, j_s)y_{kj_2}+ \ldots + D^-(j_1, \ldots, j_{s-1},l)y_{kj_s}.
\end{equation*}
\end{lm}
\begin{proof}
Using Lemma \ref{2}, the proof is straightforward.
\end{proof}

\begin{lm}
\label{3''} If $1\leq k \leq m$ and $m+1\leq l,j_1, \ldots, j_s \leq m+n$,
then $(D^-(j_1, \ldots, j_s))_{lk}D=0$.
\end{lm}
\begin{proof}
It follows immediately from Lemma \ref{2''}.
\end{proof}

\begin{lm}
\label{3'} If $m+1\leq k,l,j_1, \ldots, j_s \leq m+n$,
then $(D^-(j_1, \ldots, j_s))_{kl}D=0$ if none of the $j_i$'s equals to $k$. 
If $j_i=k$ for some $i$, then 
$(D^-(j_1, \ldots, j_s))_{kl}D=D^-(j_1, \ldots,j_{i-1}, l, j_{i+1}, \ldots j_s)$.
\end{lm}
\begin{proof}
It follows from Lemma \ref{2'}.
\end{proof}

In other words, the formulas for the action of $_{kl}D$ on $D^-(j_1, \ldots, j_s)$ for $m+1\leq k,l, j_1, \ldots, j_s \leq m+n$ 
and the action of $_{kl}D$ on $D^+(j_1, \ldots, j_s)$ for $1\leq k,l, j_1, \ldots, j_s \leq m$ are identical.

\begin{lm}
\label{3'''} Let $1\leq k,l\leq m$, where $k\neq l$, and $m+1\leq j_1, \ldots, j_s \leq m+n$. 
Then $(D^-(j_1, \ldots, j_s))_{kl}D=0$.
\end{lm}
\begin{proof}
It follows from Lemma \ref{2'''}.
\end{proof}

\subsection{Action of odd superderivations $_{kl}D$ on generators of $H^0_{G_{ev}}(\la)$}

Identify the induced module $H^0_{G_{ev}}(\lambda)$ with its image under the map $\phi$. 
Then it has a basis consisting of products $B^+(I)B^-(J)$, where $I$ has entries from the set $1, \ldots, m$ and the tableau $T(I)$ is 
standard of shape $\lambda^+$, and $J$ has entries from the set $m+1, \ldots, m+n$ and the tableau $T(J)$ is standard of shape $\lambda^-$.
Therefore the action of $_{kl}D$ on these generators of $H^0_{G_{ev}}(\la)$ is determined by
its action on determinants $D^+(i_1, \ldots, i_s)$ and $D^-(j_1, \ldots, j_s)$ which was determined earlier.

\begin{lm}\label{gen+}
Let $1\leq k\leq m$, $m+1\leq l\leq m+n$ and $w^+=B^+(I)$ be a bideterminant of shape $\lambda^+$. 
Then 
\[(w^+)_{kl}D=\sum_{a=1}^m (w^+)_{ka}D y_{al}.\]
\end{lm}
\begin{proof}
The proof follows from Lemma \ref{4}. For simplicity we derive the proof for $w^+=v^+$ and leave the general case for the reader.
 
Since $v^+$ is a primitive vector for $GL(m)$, we have $(v^+)_{ks}D=0$ for $s<k$. Also, since $v^+$ is of weight $\lambda^+$, we have 
$(v^+)_{kk}D=\lambda^+_k v^+$. Therefore we need to show that  
\[(v^+)_{kl}D=\la^+_kv^+y_{kl} +\sum_{k<s\leq m} (v^+)_{ks}D y_{sl}.\]
From Lemma \ref{4} we infer that if $k\leq j$, then 
\[(D^+(1, \ldots , j)^{\la^+_j-\la^+_{j+1}})_{kl}D=(\la^+_j-\la^+_{j+1})D^+(1, \ldots, j)^{\la^+_j-\la^+_{j+1}}y_{kl}\]
\[+ (\la^+_j-\la^+_{j+1})D^+(1, \ldots, j)^{\la^+_j-\la^+_{j+1}-1}\sum_{k<s\leq m} D^+(1, \ldots,k-1, s, k+1, \ldots j)y_{sl}\]
and $(D^+(1, \ldots , j)^{\la^+_j-\la^+_{j+1}})_{kl}D=0$ for $k>j$.  

Therefore 
\[(v^+)_{kl}D=\la^+_k v^+y_{kl} + \sum_{k<s\leq m} (\sum_{k\leq j<s} 
v^+(\la^+_j-\la^+_{j+1})\frac{D^+(1, \ldots,k-1, s, k+1, \ldots j)}{D^+(1, \ldots , j)}) y_{sl}\]
\[=\la^+_kv^+y_{kl} +\sum_{k<s\leq m} (v^+)_{ks}D y_{sl}.\]
\end{proof}

\begin{lm}\label{gen-}
Let $1\leq k\leq m$, $m+1\leq l\leq m+n$ and $w^-=B^-(J)$ be a bideterminant of shape $\lambda^-$. 
Then \[(w^-)_{kl}D=\sum_{b=m+1}^{m+n} (w^-)_{bl}D y_{kb}.\]
\end{lm} 
\begin{proof}
The proof follows from Lemma \ref{3}. For simplicity we derive the proof for $w^-=v^-$ and leave the general case for the reader.
 
Since $v^-$ is a primitive vector for $GL(n)$, we have $(v^-)_{tl}D=0$ for $t>l$. Also, since $v^-$ is of weight $\lambda^-$, we have 
$(v^-)_{ll}D=\lambda^-_l v^-$. Therefore we need to show that  
\[(v^-)_{kl}D=\la^-_lv^-y_{kl} +\sum_{m+1\leq t<l} (v^-)_{tl}D y_{kt}.\]
From Lemma \ref{3} we infer that if $m+j<l$, then 
\[(D^-(m+1, \ldots, m+j)^{\la^-_j-\la^-_{j+1}})_{kl}D=(\la^-_j-\la^-_{j+1}) D^-(m+1, \ldots, m+j)^{\la^-_j-\la^-_{j+1}-1}\times\]
\[\left[D^-(l,m+2,\ldots, m+j)y_{k,m+1}+D^-(m+1,l,\ldots, m+j)y_{k,m+2}+ \right.\]
\[\left. \ldots + D^-(m+1, \ldots, m+j-1,l)y_{k,m+j}\right]\]
and 
\[(D^-(m+1, \ldots, m+j)^{\la^-_j-\la^-_{j+1}})_{kl}D=(\la^-_j-\la^-_{j+1}) D^-(m+1, \ldots, m+j)^{\la^-_j-\la^-_{j+1}}y_{kl}\]
for $m+j\geq l$.

Therefore using Lemma \ref{3'}
\[(v^-)_{kl}D=\la^-_l v^-y_{kl}\] 
\[+ \sum_{m+1\leq t<l} 
(\sum_{t\leq m+j<l} v^-(\la^-_j -\la^-_{j+1})\frac{D^-(m+1, \ldots t-1, l, t+1, \ldots , m+j)}{D^-(m+1, \ldots, m+j)})y_{kt}\]
\[=\la^-_l v^-y_{kl} + \sum_{m+1\leq t<l} (v^-)_{tl}D y_{kt}.\]
\end{proof}

The following formula describes the action of odd superderivations $_{kl}D$ on generators $w^+w^-$ of $H^0_{G_{ev}}(\lambda)$.

\begin{teo}\label{gen}
Let $1\leq k\leq m$, $m+1\leq l\leq m+n$, $w^+=B^+(I)$ be a bideterminant of shape $\lambda^+$,  
$w^-=B^-(J)$ be a bideterminant of shape $\lambda^-$ and $w=w^+w^-$.
Then 
\[(w)_{kl}D=\sum_{a=1}^m (w)_{ka}D y_{al}+\sum_{b=m+1}^{m+n} (w)_{bl}D y_{kb} \text{ and } (w)_{lk}D=0.\]
\end{teo}
\begin{proof}
The first formula follows from Lemmas \ref{gen-} and \ref{gen+}. The statement that $(w)_{lk}D=0$ folows from observation that the weight 
of the element $(w)_{lk}D\in H^0_G(\la)$ does not belong to the weights of $H^0_G(\la)$. These weights are described with the help of the isomorphism $\tilde{\phi}$ 
or one can use Corollary 5.4 of \cite{z}.
\end{proof}

Note that if the weight of $w^+$ is $\mu^+$ and the weight of $w^-$ is $\mu^-$, then the terms corresponding to $a=k$ and $b=l$ in the above formula
add up to $(w)_{kk}Dy_{kl}+(w)_{ll}Dy_{kl}=(\mu^+_k+\mu^-_l)wy_{kl}$.  
\begin{cor}\label{genv}
Let $1\leq k \leq m, m+1\leq l\leq m+n$ and $v=v^+v^-$ be the highest vector in $H^0_G(\lambda)$ defined earlier. Then
\[(v)_{kl}D=(\la^+_k+\la^-_l)vy_{kl} +\sum_{k<s\leq m} (v)_{ks}D y_{sl} + \sum_{m+1\leq t<l} (v)_{tl}D y_{kt} \text{ and } (v)_{lk}D=0.\]
\end{cor}

\subsection{Action of even superderivations $_{kl}D$, $_{kl}^{(r)}D$ and $\binom{_{kk}D}{r}$ on generators of $H^0_{G_{ev}}(\la)$}

\begin{lm} Let $w^+=B^+(I)$ be a bideterminant of shape $\lambda^+$,  $w^-=B^-(J)$ be a bideterminant of shape $\lambda^-$, $w=w^+w^-$ and $r\geq 1$.
If $1\leq k,l \leq m$, then $(w)_{kl}^{(r)}D=(w^+)_{kl}^{(r)}Dw^-$ and $(w)\binom{_{kk}D}{r}=(w^+)\binom{_{kk}D}{r}w^-$. 
If $m+1\leq k,l \leq m+n$, then $(w)_{kl}^{(r)}D=w^+(w^-)_{kl}^{(r)}D$ and $(w)\binom{_{kk}D}{r}=w^+(w^-)\binom{_{kk}D}{r}$.
\end{lm}
\begin{proof}
Let us compute inside the $\mathbb{Z}$-form $\mathbb{C}[G]_{\mathbb{Z}}$ first. 

If  $1\leq k,l \leq m$, then Lemma \ref{3'''} implies that $(w^-)_{kl}D=0$. Therefore $(w^-)_{kl}^{(r)}D=0$ and $(w^-)\binom{_{kk}D}{r}=0$ for each $r>0$.
This implies $(w)_{kl}D=(w^+)_{kl}Dw^-$, $(w)_{kl}^{(r)}D=(w^+)_{kl}^{(r)}Dw^-$ and $(w)\binom{_{kk}D}{r}=(w^+)\binom{_{kk}D}{r}w^-$.

If  $m+1\leq k,l \leq m+n$, then Lemma \ref{4''} implies that $(w^+)_{kl}D=0$. Therefore $(w^+)_{kl}^{(r)}D=0$ and $(w^+)\binom{_{kk}D}{r}=0$ for each $r>0$.
This implies $(w)_{kl}D=w^+(w^-)_{kl}D$, $(w)_{kl}^{(r)}D=w^+(w^-)_{kl}^{(r)}D$ and $(w)\binom{_{kk}D}{r}=w^+(w^-)\binom{_{kk}D}{r}$.

When we apply the modular reduction to the base field $K$ of characteristic $p>2$, all the above formulas remain valid.
\end{proof}

Using the previous Lemma we can simplify actions of even superderivations and their divided powers on $H^0_{G_{ev}}(\la)$ as follows. 
If $1\leq k,l \leq m$, then $_{kl}D$, $_{kl}^{(r)}D$ and $\binom{_{kk}D}{r}$ act as their classical counterparts on $H^0_{GL(m)}(\la^+)$ and act trivially on $H^0_{GL(n)}(\la^-)$.
If $m+1\leq k,l\leq m+n$, then $_{kl}D$, $_{kl}^{(r)}D$ and $\binom{_{kk}D}{r}$ act as images under $\phi$ of their classical counterparts on $H^0_{GL(n)}(\la^-)$ 
and act trivially on $H^0_{GL(m)}(\la^+)$.
Thus the action of even $_{kl}D$, $_{kl}D^{(r)}$ and $\binom{_{kk}D}{r}$ on $H^0_{G_{ev}}(\la)$ is determined by the $GL(m)$-structure of $H^0_{GL(m)}(\la^+)$ and the $GL(n)$-structure of 
$H^0_{GL(n)}(\la^-)$.

\section{Primitive vectors in $H^0_{G_{ev}}(\la)\otimes Y$}

We will consider $H^0_{G_{ev}}(\lambda)$ embedded into $H^0_G(\lambda)$ via the map $\phi$ as before.

In order to study the structure of the induced supermodule $H^0_G(\la)$, we will consider the tensor products 
$F_r=H^0_{G_{ev}}(\la)\otimes (\wedge^r Y)$ of $G_{ev}$-module $H^0_{G_{ev}}(\la)$ with 
the $r$-th exterior power $\wedge^i Y= Y \wedge Y \ldots \wedge Y$ of the $G_{ev}$-module $Y$.

The module $F_r$, that will be called the $r$-th floor, is given as a $K$ -span of all vectors 
$w\otimes (y_{i_1,j_1}\wedge\ldots \wedge y_{i_r,j_r})$, where $w\in H^0_{G_{ev}}(\la)$ and elements $y_{i_1, j_1}, \ldots, y_{i_r, j_r}$ are pairwise distinct.

It is clear that $H^0_G(\lambda)$ viewed as a $G_{ev}$-module is isomorphic to the direct sum $\oplus_{r=0}^{mn} F_r$ and when trying to determine the $G_{ev}$-primitive vectors we can 
restrict our attention to individual floors $F_r$. 

In this section, we will describe explicitly $G_{ev}$-primitive vectors in $F_1=H^0_{G_{ev}}(\la)\otimes Y$. 
In characteristic zero case we will provide a complete list of such primitive vectors.

\subsection{Good filtration of $H^0_{G_{ev}}(\la)\otimes Y$}

Earlier we have described the $G_{ev}$-module structure of $Y$ as $Y\simeq (V_m)^*\otimes V_n$, where $V_m$ is the natural representation of $GL(m)$ and $V_n$ is the natural representation of $GL(n)$.

Since $H^0_{G_{ev}}(\la)\cong H^0_{\GL(m)}(\la^+)\otimes H^0_{\GL(n)}(\la^-)$, we can identify the $G_{ev}$-module 
$H^0_{G_{ev}}(\la)\otimes Y$ with the tensor product of the $\GL(m)$-module $H^0_{\GL(m)}(\la^+)\otimes (V_m)^*$ and 
the $\GL(n)$-module $H^0_{\GL(n)}(\la^-)\otimes V_n$. 

For $1\leq i \leq m$ denote by $\delta^+_i$ the weight for $G$ that has all entries equal to zero except the entry at the $i$th position which equals one, and for $1\leq j\leq n$
denote by $\delta^-_j$ the weight for $G$ that has all entries equal to zero except the entry at the $(m+j)$th position which equals one.
We will abuse the notation and denote by $\delta^+_i$ the corresponding weight of $GL(m)$ and by $\delta^-_j$ the corresponding weight of $GL(n)$.
Also, for $1\leq i\leq m$ and $1\leq j\leq n$ denote $\lambda_{ij}= \lambda-\delta^+_i+\delta^-_j$.

The following lemma is the folklore for which we were unable to find an appropriate reference.

\begin{lm}\label{first-floor}
The $G_{ev}$-module $H^0_{G_{ev}}(\la)\otimes Y$ has a (good) filtration by $G_{ev}$-induced modules $H^0_{G_{ev}}(\la_{ij})$.
Each factor in this filtration has a simple multiplicity and $H^0_{G_{ev}}(\la_{ij})$ appears in this filtration if and only if 
$\la^+_{i}\neq \la^+_{i+1}$ if $i<m$, $\la^+_i\neq 0$ for $i=m$, and $\la^-_{j-1}\neq \la^-_{j}$ if $j>1$.
\end{lm}
\begin{proof}
First we consider the case of characteristic zero. 

The structure of the $GL(n)$-module $H^0_{\GL(n)}(\la^-)\otimes V_n$ is given by the Littlewood-Richardson rule as
a direct sum of simple $GL(n)$-modules with highest weights $\la^-+\delta^-_{j}$, where the index $j$ is such that $\la^-_{j-1}\neq \la^-_{j}$ if $j>1$.

On the other hand, $V_m^*\simeq \Lambda^{m-1}(V_m)\otimes D^{-1}$. 
The Littlewood-Richardson rule shows that $H^0_{GL(m)}(\la^+) \otimes \Lambda^{m-1}(V_m)$ is a direct sum of simple $GL(m)$-modules with highest weights
$\la^+ +\sum_{1\leq t \leq m, t\neq i} \delta^+_t$, where $1\leq i\leq m$ is such that $\la^+_{i+1}\neq \la^+_i$ if $i<m$ and $\la^+_i\neq 0$ if $i=m$.
Since the (highest) weight of $D^{-1}$ is equal to
$-(\sum_{1\leq t\leq n} \delta^+_t)$, we conclude that $H^0_{\GL(m)}(\la^+)\otimes (V_m)^*$ is a direct sum of simple $GL(m)$-modules with 
highest weights $\la^+-\delta^+_{i}$, where the index $i$ is such that $\la^+_{i+1}\neq \la^+_i$ if $i<m$ and $\la^+_i\neq 0$ if $i=m$.
 
Therefore the $G_{ev}$-module $H^0_{G_{ev}}(\la)\otimes Y$ is a direct sum of $G_{ev}$-induced modules of highest weight $\la_{ij}$ satisfying the given restrictions on indices
$i$ and $j$.

Now assume that the characteristic of the base field $K$ is positive.

Denote by $1^s$ the weight corresponding to the partition consisting of $s$ terms which are all equal to $1$.
Since $V_n\simeq H^0_{GL(n)}(1)$, by Donkin-Mathieu theorem, the module $H^0_{\GL(n)}(\la^-)\otimes V_n$ has a filtration by induced $GL(n)$-modules.
Since $\Lambda^{m-1}(V_m)\simeq H^0_{GL(m)}(1^{m-1})$, again by Donkin-Mathieu theorem, the module $H^0_{GL(m)}(\la^+) \otimes \Lambda^{m-1}(V_m)$ has a filtration 
by induced $GL(m)$-modules. Since tensoring with $D^{-1}$ is not going to influence anything, we conclude that $H^0_{\GL(m)}(\la^+)\otimes (V_m)^*$ has a filtration
by induced $GL(m)$-modules. 
Consequently, the tensor product $H^0_{G_{ev}}(\la) \otimes Y$ has a filtration by induced $G_{ev}$-modules. Since the formal characters of induced modules and exterior powers
do not depend on the characteristic, the factors in the good filtration of $H^0_{G_{ev}}(\la)\otimes Y$ must be the same as in the characteristic zero case.
\end{proof}

\subsection{Case $n=1$}
If $n=1$, then $v^-=D^-(m+1)^{\la^-_1}=\phi(c_{m+1,m+1})^{\la^-_1}$.

For $i=1, \ldots, m-1$ such that $\lambda^+_i\neq \lambda^+_{i+1}$ and $i=m$ such that $\la^+_m\neq 0$
denote $f^+_i=\frac{v^+}{D^+(1,\ldots, i)}$,
and $g^+_i=D^+(1, \ldots, i) _{i,m+1}D$. 

If $f^+_i$ is defined, then conditions $\lambda^+_i-\lambda^+_{i+1}> 0$ for $i<m$ and $\la^+_i >0$ for $i=m$ imply that $f^+_i$ is a product of determinants $D^+(1,\ldots, j)$.

\begin{lm}\label{n=1} Assume $1\leq i <m$ and $\la^+_i-\la^+_{i+1}\neq 0$, or $i=m$ and $\la^+_i\neq 0$. Then the vector
\[\pi^+_i=\frac{v^+}{D^+(1,\ldots, i)}\sum_{k=i}^m D^+(1, \ldots, i-1,k)y_{k,m+1}\phi(c_{m+1,m+1})^{\la^-_1}\] 
is a $G_{ev}$-primitive vector of $F_1$.

If $char(K)=0$, then the set of all vectors $\pi^+_i$ as above forms a complete set of $G_{ev}$-primitive vectors in $F_1$. 
Consequently, $F_1$ is a direct sum of irreducible $G_{ev}$-modules with highest vectors $\pi^+_i$. 
\end{lm}
\begin{proof}
It is clear that $(f^+_i)_{l+1,l}D=0$ for every $l=1, \ldots, m-1$.

We will verify that $(g^+_i)_{l+1,l}D=0$ for every $l=1, \ldots m-1$.
By Lemma \ref{4} we infer that $g^+_i=\sum_{k=i}^m D^+(1, \ldots, i-1,k)y_{k,m+1}$. If $l\leq i-1$, then each $D^+(1, \ldots, i-1,k)_{l+1,l}D=0$
and each $(y_{k,m+1})_{l+1,l}D=0$ implies $(g^+_i)_{l+1,l}D=0$. If $l\geq i$, then $(g^+_i)_{l+1,l}D=D^+(1, \ldots, i-1,l)(-y_{l+1,m+1})+ D^+(1, \ldots, i-1,l)y_{l+1,m+1}=0$.

Therefore $(f^+_i)_{kl}D=0$, $(f^+_i)_{kl}D^{(r)}=0$ as well as $(g^+_i)_{kl}D=0$, $(g^+_i)_{kl}D^{(r)}=0$
for every $1\leq l<k\leq m$ and $r>0$.
For $n=1$ the group $GL(n)$ is one-dimensional torus and thus
$\pi^+_i=f^+_ig^+_i\phi(c_{m+1,m+1})^{\la^-_1}$ is a primitive $G_{ev}$-element.
Because every $v^+\frac{D^+(1, \ldots, i-1,k)}{D^+(1,\ldots, i)}= B_+(1^{\lambda^+_1} \ldots i^{\lambda^+_i-1}k(i+1)^{\lambda^+_{i+1}} \ldots m^{\lambda^+_m})$ 
is a bideterminant of shape $\lambda^+$, it belongs to $H^0_{GL(m)}(\la^+)$. This implies that $\pi^+_i$ belongs to $F_1$. 

If $char(K)=0$, then these vectors $\pi^+_i$ have different $G_{ev}$-weights and modules $H^0_{G_{ev}}(\pi^+_i)$ 
are composition factors in the filtration of $F_1$ from Lemma \ref{first-floor}, hence $\pi^+_i$ form a complete set of $G_{ev}$-primitive vectors in $F_1$ and $F_1$ is a direct sum 
of irreducible modules $H^0_{G_{ev}}(\pi^+_i)$.
\end{proof}

\subsection{Case $m=1$}

If $m=1$, then $v^+=D^+(1)^{\la^+_1}=c_{1,1}^{\la^+_1}$.

Denote $f^-_1=v^-$, $g^-_1=y_{m,m+1}$ and for $1< j \leq n$ such that $\lambda^-_{j-1}-\lambda^-_{j}\neq 0$ denote $f^-_j=\frac{v^-}{D^-(m+1,\ldots, m+j-1)}$ and 
\[g^-_j=\sum_{k=1}^{j} (-1)^{k+j} D^-(m+1, \ldots, \widehat{m+k}, \ldots, m+j)y_{m,m+k}.\]

If $f^-_j$ is defined, then the conditions $\lambda^-_{j-1}-\lambda^-_{j}> 0$ for $1< j \leq n$ imply that $f^-_j$ is a product of determinants $D^-(m+1,\ldots, m+i)$. 

In order to simplify the notation, define $D^-(m+1, \ldots, m+j-1)=1$ for $j=1$ and $D^-(m+1, \ldots, \widehat{m+k},\ldots, m+j)=1$ for $k=j=1$.

\begin{lm}\label{m=1} Assume $1\leq j\leq n$ is such that $\la^-_{j-1}\neq \la^-_{j}$ if $j>1$.
Then the vector
\[\pi^-_j=c_{11}^{\lambda^+_1}\frac{v^-}{D^-(m+1,\ldots, m+j-1)}\sum_{k=1}^{j} (-1)^{k+j} D^-(m+1, \ldots, \widehat{m+k},\ldots, m+j)y_{1,m+k}\] 
is a $G_{ev}$-primitive vector of $F_1$.

If $char(K)=0$, then the set of all vectors $\pi^-_j$ as above forms a complete set of $G_{ev}$-primitive vectors in $F_1$. 
Consequently, $F_1$ is a direct sum of irreducible $G_{ev}$-modules with highest vectors $\pi^-_j$. 
\end{lm}
\begin{proof}
It is clear that $(f^-_j)_{l+1,l}D=0$ for every $l=m+1, \ldots m+n-1$ and $(g^-_j)_{l+1,l}D=0$ for $l>j+1$.

If $m<l\leq j+1$, then Lemma \ref{1''} and \ref{3'} imply that 
\[(g^-_j)_{l+1,l}D=(-1)^{l+j+1}D^-(m+1, \ldots, \widehat{l+1}, \ldots, m+j)(y_{1,l+1})_{l+1,l}D\]
\[+(-1)^{l+j}(D^-(m+1, \ldots, \hat{l}, \ldots, m+j))_{l+1,l}Dy_{1,l}\]
\[=(-1)^{l+j+1}D^-(m+1, \ldots, \widehat{l+1}, \ldots, m+j)y_{1,l}+\]
\[(-1)^{l+j}D^-(m+1, \ldots, \widehat{l+1}, \ldots, m+j)y_{1,l}=0.\]

Therefore $(f^-_j)_{kl}D=0$, $(f^-_j)_{kl}D^{(r)}=0$ as well as $(g^-_j)_{kl}D=0$, $(g^-_j)_{kl}D^{(r)}=0$
for every $m+1\leq l<k\leq m+n$ and $r>0$.
For $m=1$ the group $GL(m)$ is one-dimensional torus and thus $\pi^-_j=c_{11}^{\lambda^+_1}f^-_jg^-_j$ is a primitive $G_{ev}$-element. 
Every $v^-\frac{D^-(m+1, \ldots, \widehat{m+k},\ldots, m+j)}{D^-(m+1,\ldots, m+j-1)}= \phi(B_-(J_{jk}))$,
where $J_{jk}$ is such that $T^-(J_{jk})$ is obtained from $T^-_{\lambda}$ by replacing one column with entries $m+1, \ldots, m+j-1$ with the column with entries
$m+1, \ldots, \widehat{m+k},\ldots, m+j$. Since each $v^-\frac{D^-(m+1, \ldots, \widehat{m+k},\ldots, m+j)}{D^-(m+1,\ldots, m+j-1)}$ 
is a bideterminant of shape $\lambda^-$, it belongs to $H^0_{GL(n)}(\la^-)$. This implies that $\pi^-_i$ belongs to $F_1$. 

If $char(K)=0$, then these vectors $\pi^-_j$ have different $G_{ev}$-weights and modules $H^0_{G_{ev}}(\pi^-_j)$ 
are composition factors in the filtration of $F_1$ from Lemma \ref{first-floor}, hence $\pi^-_j$ form a complete set of $G_{ev}$-primitive vectors in $F_1$ and $F_1$ is a direct sum 
of irreducible modules $H^0_{G_{ev}}(\pi^-_j)$.
\end{proof} 

\subsection{General case}

To deal with the general case, we first describe explicitly the isomorphism $V_m^*\otimes V_n \simeq Y$ of $G_{ev}$-modules discussed earlier. Let $(v^+_m)^*$ be the highest vector of 
$V_m^*$ and define inductively $(v^+_k)^*= (v^+_{k+1})^*_{k+1,k}D$ for $k=m-1, \ldots , 1$. Let $v_n$ be the highest vector of $V_n$ and define 
inductively $v^-_k=(v^-_{k+1})_{k+1,k}D$
for $k=n-1, \ldots , 1$. Using Lemmas \ref{1''} and \ref{1'''} we verify that the isomorphism $V_m^*\otimes V_n \simeq Y$ is given via 
$(v^+_i)^*\otimes v^-_j \mapsto (-1)^{m+i} y_{i,m+j}$.

Recall our previous convention that $D^-(m+1, \ldots, m+j-1)=1$ for $j=1$ and $D^-(m+1, \ldots, \widehat{m+k},\ldots, m+j)=1$ for $k=j=1$.

\begin{pr}\label{prim in F1} Assume $1\leq i <m$ and $\la^+_i\neq \la^+_{i+1}$, or $i=m$ and $\la^+_i\neq 0$. 
Assume $1\leq j\leq n$ is such that $\la^-_{j-1}\neq \la^-_{j}$ if $j>1$.
Then the vector
\[\pi_{ij}=\frac{v}{D^+(1,\ldots, i)D^-(m+1,\ldots, m+j-1)}\sum_{r=i}^m D^+(1, \ldots, i-1,r)\]
\[\sum_{s=1}^{j} (-1)^{s+j} D^-(m+1, \ldots, \widehat{m+s},\ldots, m+j)y_{r,m+s}\]
is a $G_{ev}$-primitive vector of $F_1$.

If $char(K)=0$, then the set of all vectors $\pi_{ij}$ as above forms a complete set of $G_{ev}$-primitive vectors in $F_1$.
Consequently, $F_1$ is a direct sum of irreducible $G_{ev}$-modules with highest vectors $\pi_{ij}$.
\end{pr}
\begin{proof}
The primitive vectors of the $\GL(m)$-module $H^0_{\GL(m)}(\la_+)\otimes (V^+_m)^*$ correspond to primitive vectors of the $\GL(m)\times \GL(1)$-supermodule
$H^0_{\GL(m|1)_{ev}}(\la)\otimes Y_{\GL(m|1)}$, where $Y_{\GL(m|1)}$ is the previously defined $Y$ in the special case $n=1$.
 Using Lemma \ref{n=1} we infer that they are given as
\[\frac{v^+}{D^+(1,\ldots, i)}\sum_{r=i}^m D^+(1, \ldots, i-1,r)(v^+_r)^*(-1)^{m+r}\]
whenever $1\leq i <m$ and $\la^+_i\neq \la^+_{i+1}$, or $i=m$ and $\la^+_i\neq 0$. 

The primitive vectors of the $\GL(n)$-module $H^0_{\GL(n)}(\la_-)\otimes (V^-_n)$ correspond to primitive vectors of the $\GL(1)\times \GL(n)$-supermodule
$H^0_{\GL(1|n)_{ev}}(\la)\otimes Y_{\GL(1|n)}$, where $Y_{\GL(1|n)}$ is the previously defined $Y$ in the special case $m=1$.
Using Lemma \ref{m=1} we infer that they are given as
\[\frac{v^-}{D^-(m+1,\ldots, m+j-1)}\sum_{s=1}^{j} (-1)^{s+j} D^-(m+1, \ldots, \widehat{m+s},\ldots, m+j)v^-_s\]
whenever $1\leq j\leq n$ is such that $\la^-_{j-1}\neq \la^-_{j}$ if $j>1$.

The corresponding tensor product 
\[(\frac{v^+}{D^+(1,\ldots, i)}\sum_{r=i}^m D^+(1, \ldots, i-1,r)(v^+_r)^*(-1)^{m+r})\]
\[\otimes(\frac{v^-}{D^-(m+1,\ldots, m+j-1)}
\sum_{s=1}^{j} (-1)^{s+j} D^-(m+1, \ldots, \widehat{m+s},\ldots, m+j)v^-_s)\]
is a $G_{ev}$-primitive element of $H^0_{G_{ev}}(\la)\otimes Y$ which is identified with $\pi_{ij}\in F_1$.

If $char(K)=0$, then these vectors $\pi_{ij}$ have different $G_{ev}$-weights and modules $H^0_{G_{ev}}(\pi_{ij})$ 
are composition factors in the filtration of $F_1$ from Lemma \ref{first-floor}, hence $\pi_{ij}$ form a complete set of $G{ev}$-primitive vectors in $F_1$ and $F_1$ is a direct sum 
of irreducible modules $H^0_{G_{ev}}(\pi_{ij})$.
\end{proof}

\subsection{Images of primitive vectors under $\phi_1$}

Consider the map $\phi_1:F_1 \to F_1$ defined as 
$\phi_1(w\otimes y_{ij})=(w)_{ij}D$ for $w\in H^0_{G_{ev}}(\la)$ for $1\leq i\leq m$ and $m+1\leq j \leq m+n$.

\begin{lm}\label{mor1}
The map $\phi_1$ is a morphism of $G_{ev}$-modules.
\end{lm}
\begin{proof}
We will check that $\phi_1$ preserves the action of even superderivations $_{kl}D$.

Assume either $1\leq k,l \leq m$ or $m+1\leq k,l \leq m+n$.
Then $(\phi_1(w\otimes y_{ij}))_{kl}D = (w)_{ij}D_{kl}D $. On the other hand, Lemmas \ref{1''} and \ref{1'''} imply
\[(w\otimes y_{ij})_{kl}D= (w)_{kl}D\otimes y_{ij} - \delta_{li} w\otimes y_{kj}+\delta_{jk} w \otimes y_{il}.\]
Therefore 
\[\phi_1((w\otimes y_{ij})_{kl}D)=(w)_{kl}D_{ij}D- \delta_{li} (w)_{kj}D+\delta_{jk} (w)_{il}D.\]
Since $_{ij}D_{kl}D-\,_{kl}D_{ij}D=\delta_{jk} \, _{il}D -\delta_{li} \, _{kj}D$, we conclude that 
$(\phi_1(w\otimes y_{ij}))_{kl}D = \phi_1((w\otimes y_{ij})_{kl}D)$.
\end{proof}

To compute images of $\pi_{ij}$ under $\phi_1$ we will use Lemmas \ref{3} and \ref{4} repeatedly.

Denote by $\lhd$ the usual ordering on weights for $G_{ev}$ such that $\mu \lhd \lambda$ if and only if $\lambda-\mu$ is a sum of 
simple roots of $G_{ev}$.
The vector $\pi_{ij}$ is a sum of the term $vy_{i,m+j}$, that we will call the leading term, and other terms that are multiples of $w_{rs}\otimes y_{r,m+s}$, 
where the weight $\mu$ of $w_{rs}\in H^0_{G_{ev}}(\la)$ satisfies $\mu\neq \la$ and $\mu\lhd\la$, that we will call lower terms.  
Hence $\pi_{ij}=vy_{i,m+j}+$ lower terms. 

\begin{pr}\label{phi1}
Assume $1\leq i <m$ and $\la^+_i\neq \la^+_{i+1}$, or $i=m$ and $\la^+_i\neq 0$; and $1\leq j\leq n$ is such that $\la^-_{j-1}\neq \la^-_{j}$ if $j>1$.
Then $\phi_1(\pi_{ij})=\omega_{ij}\pi_{ij}$.
\end{pr}
\begin{proof}
Using Lemmas \ref{3} and \ref{4} we compute $(v)_{i,m+j}D=(\la^+_i+\la^-_j)vy_{i,m+j}+$ lower terms.
For the image of $(w_{rs})_{r,m+s}D$, where 
\[w_{rs}=\frac{(-1)^{s+j}vD^+(1, \ldots, i-1, r)D^-(m+1, \ldots, \hat{s}, \ldots m+j)}{D^+(1, \ldots, i)D^-(m+1, \ldots, m+j-1)},\]
to contain a nonzero multiple of the leading term $vy_{i,m+j}$, we must have that 
\[(D^+(1, \ldots, i-1, r)D^-(m+1, \ldots, \hat{s}, \ldots m+j))_{r,m+s}D\]
is a multiple of $D^+(1, \ldots, i)D^-(m+1, \ldots,  m+j-1)$ and this can happen only if $r=i$ or $s=m+j$. 
If $r=i$ and $s<m+j$, then this coefficient equals $-1$; if $r>i$ and $s=m+j$, this coefficient equals $+1$.
Therefore $vy_{i,m+j}$ appears in $\phi_1(\pi_{ij})$ with the coefficient $\la^+_i+\la^-_j-(j-1)+(m-i)=\la^+_i+\la^-_j+m+1-i-j=\omega_{i,j}$.
Since $\pi_{ij}$ is a primitive $G_{ev}$-vector and the map $\phi_1$ is a $G_{ev}$-morphism by Lemma \ref{mor1}, the claim follows.
\end{proof}
 
\begin{cor} \label{cor}
Assume $1\leq i <m$ and $\la^+_i\neq \la^+_{i+1}$, or $i=m$ and $\la^+_i\neq 0$; and $1\leq j\leq n$ is such that $\la^-_{j-1}\neq \la^-_{j}$ if $j>1$.
If $char(K)=0$ and $\omega_{ij}=0$, then $H^0_G(\la)$ is not irreducible. If $char(K)=p>2$ and $\omega_{ij} \equiv 0 \pmod p$, 
then $H^0_G(\la)$ is not irreducible
\end{cor}
\begin{proof}
The simple $G$-module $L_G(\la)$ is the socle of $H^0_G(\la)$ and is generated by the highest weight vector $v$ of weight $\la$. We will show that 
the intersection $L_G(\la)\cap F_1$ equals the image of the map $\phi_1$. 
Use Poincare-Birkhoff-Witt theorem and order the superderivations in the following way. Start with $_{ij}D$ and $_{ij}^{(r)}D$ with $i>j$, continue with odd $_{ij}D$ with $i<j$ followed by even $_{ij}D$, $_{ij}^{(r)}D$ for $i<j$ and $\binom{_{ii}D}{r}$.
The vector $v$ is annihilated by $_{ij}D$ and $_{ij}^{(r)}D$ where $i>j$. Each odd superderivation $_{ij}D$, where $i<j$, sends a vector from floor $F_k$ to the next floor $F_{k+1}$ and each even $_{ij}D$, $_{ij}^{(r)}D$ for $i<j$ and $\binom{_{ii}D}{r}$ sends $F_k$ to itself. Therefore $L_G(\la)\cap F_1$ is the sum of $G_{ev}$-modules generated by $(v)_{ij}D$ for
$1\leq i\leq m$ and $m+1\leq j \leq m+n$. Using Lemma \ref{mor1} we conclude that this is exactly the image of the map $\phi_1$.

If $char(K)=0$ and $\omega_{ij}=0$, or $char(K)=p>2$ and $\omega_{ij} \equiv 0 \pmod p$, then $\phi_1(\pi_{ij})=0$ by Proposition \ref{phi1}. 
Since $\phi_1$ is not injective, its image is a proper subset of $F_1$. Therefore $H^0_G(\la)$ is not irreducible.
\end{proof}

\begin{rem}
The above corollary is a special case of the statement characterizing irreducible induced modules $H^0_G(\la)$ in arbitrary characteristic. In the upcoming paper we will show that 
$H^0_G(\la)$ is irreducible if and only if $\la$ is typical.

\end{rem}

\section{Primitive vectors in $H^0_{G_{ev}}(\la) \otimes \wedge^k Y$}\label{s4}

The purpose of this section is to investigate advantages and limitations of extending the previous construction of $G_{ev}$-primitive vectors to higher floors $F_k$.  

In this section we will assume that indices $i$ take on values from the set $\{1, \ldots, m\}$ and indices $j$ take on values from the set 
$\{1, \ldots, n\}$.

\subsection{Primitive vectors in $F^\otimes_k=H^0_{G_{ev}}(\la) \otimes Y^{\otimes k}$}

Assume that $(I|J)=(i_1\ldots i_k|j_1\ldots j_k)$ is a multi-index such that $1\leq i_1, \ldots, i_k \leq m$ and $1\leq j_1, \ldots, j_k\leq n$. 
Define the content $cont(I|J)$ of $(I|J)$ to be the $(m+n)$-tuple 
$cont(I|J)=(x^+_1, \ldots, x^+_m|x^-_{1}, \ldots, x^-_{n})$, where 
$x^+_s$ is the negative of the number of occurences of the symbol $s$ in $i_1\ldots i_k$ and 
$x^-_t$ is the number of occurences of the symbol $t$ in $j_1\ldots j_k$. 
Further, denote
\[\la_{I|J}=\la-\sum_{s=1}^k \delta^+_{i_s}+\sum_{s=1}^k \delta^-_{j_s}.\]

For each $1\leq i \leq m$ and $1\leq j\leq n$ denote by $\rho_{i|j}$ the following element 
\[\sum_{r=i}^m D^+(1, \ldots, i-1,r)\sum_{s=1}^{j} (-1)^{s+j} D^-(m+1, \ldots, \widehat{m+s},\ldots, m+j)y_{r,m+s}\]
of $A(m|n)\otimes Y$. Here we set $D^-(m+1, \ldots, \widehat{m+s},\ldots, m+j)=1$ for $s=j=1$.
For each $(I|J)=(i_1\ldots i_k|j_1\ldots j_k)$ such that $1\leq i_1, \ldots, i_k \leq m$ and $1\leq j_1, \ldots, j_k\leq n$
denote $\rho_{I|J}=\otimes_{s=1}^k \rho_{i_s|j_s}$. We will abuse the notation and consider $\rho_{I|J}$ as an element of $A(m|n)\otimes Y^{\otimes k}$
by changing the order of terms in $\otimes_{s=1}^k \rho_{i_s|j_s}$.

Consider the following elements
\[v_{I|J}=\frac{v}{\prod_{s=1}^k D^+(1, \ldots, i_s)\prod_{s=1}^k D^-(m+1,\ldots, m+j_s-1)},\]
where we set $D^-(m+1,\ldots, m+j_s-1)=1$ for $j_s=1$.

By definition, $v_{I|J}$ belongs to the quotient field $K(m|n)$ of the superalgebra $A(m|n)$. The element $v_{I|J}$ belongs to $A(m|n)$ if and only if 
$\la^+_s-\la^+_{s+1}\geq -cont(I|J)^+_s$ for $s=1, \ldots, m-1$, 
$\la^+_m\geq -cont(I|J)^+_m$,
and 
$\la^-_{t-1}-\la^-_{t}\geq cont(I|J)^-_t$ for $t=2, \ldots, n$.
In other words, the symbol $i_s<m$ appears at most $\la^+_{i_s}-\la^+_{i_s+1}$ times in $I$,
symbol $m$ appears at most $\la^+_m$ times in $I$; and 
symbol $j_t>1$ appears at most $\la^-_{j_t-1}-\la^-_{j_t}$ times in $J$.

The weight $\la$ will be called {\it $(I|J)$-robust} if $v_{I|J}\in A(m|n)$.  Clearly, if $cont(K|L)=cont(I|J)$, then $\la$ is $(I|J)$-robust if and only if it is $(K|L)$-robust. 

Finally, denote \[\pi_{I|J}=v_{I|J}\rho_{I|J}\] an element of weight $\la_{I|J}$. 
The element $\pi_{I|J}$ belongs to $K(m|n)\otimes Y^{\otimes k}$ and does not necessarily lie in $F^\otimes_k$. 
However, if $\la$ is $(I|J)$-robust, then $\la_{I|J}$ is dominant and $\pi_{I|J}\in F^\otimes_k$.

\begin{lm}\label{Fotimes}
Assume the element $\pi_{I|J}$ costructed above is such that $\la_{I|J}$ is dominant and $\pi_{I|J}\in F^\otimes_k$. Then $\pi_{I|J}$ is $G_{ev}$-primitive element of $F^\otimes_k$.
In particular, this is the case when $\la$ is $(I|J)$-robust.
\end{lm}
\begin{proof}
As in the previous section, we check that $v_{I|J}$ and each $\rho_{ij}$, hence also $\rho_{I|J}$, are annihilated by $_{kl}D$ and $_{kl}D^{(r)}$ for every 
$1\leq k<l\leq m$ or $m+1\leq k<l \leq m+n$ and $r>1$. Therefore $\pi_{I|J}$ is a $G_{ev}$-primitive element of $F^\otimes_k$.
\end{proof}

\subsection{Primitive vectors in $F_k$}

To proceed to the space $F_k=H^0_{G_{ev}}(\la) \otimes (\wedge^k Y)$, we will consider the usual morphism $\otimes^k Y \to \wedge^k Y$ which induces the 
morphism $M:F^\otimes_k \to F_k$. Some images $M(\pi_{I|J})$ vanish and others coincide. We would like to choose primitive vectors 
in $F^\otimes_k$ in such a way that their images under $M$ form a basis of all primitive vectors generated by $M(\pi_{I|J})$. 
In doing so we can split the problem into individual weight spaces in $F_k$. Unlike the case of $F_1$, where the multiplicities of primitive vectors 
were simple, we will deal with many primitive vectors of the same weight, see \cite{gm1}.   

Assume $(I|J)=(i_1\ldots i_k|j_1\ldots j_k)$ is such that $\la_{I|J}$ is dominant,
$i_1\leq i_2 \ldots \leq i_k$ and $i_r=i_{r+1}$ implies $j_r<j_{r+1}$. Such $(I|J)$ will be called \emph{admissible}. 


\begin{lm}\label{exterior}
Let $(I|J)=(i_1\ldots i_k|j_1\ldots j_k)$ be admissible and $\pi_{I|J}\in F^{\otimes k}$.
Then $M(\pi_{I|J})$ is a nonzero $G_{ev}$-primitive vector in $F_k$.
\end{lm}
\begin{proof}
By Lemma \ref{Fotimes}, the element $M(\pi_{I|J})$ is $G_{ev}$-primitive vector in $F_k$. It is nonzero because $(I|J)$ is admissible.
\end{proof}

From now on we will work only with admissible $(I|J)$ and identify $M(\pi_{I|J})$ with $\pi_{I|J}$.

For an admissible $(I|J)=(i_1\ldots i_k|j_1\ldots j_k)$ define the height $ht(I|J)=\sum_{s=1}^k j_s-i_s$ and denote $y_{I|J}=y_{i_1,m+j_1} \ldots y_{i_k, m+j_k}$.

Each vector $x\in H^0_G(\la)$ of weight $\mu$ can be written as
$x=\sum_{(I|J)} x_{I|J}y_{I|J}$, where the sum is over admissible $(I|J)$, $x_{I|J}\in H^0_{G_{ev}}(\la)$ is of weight $\gamma_{I|J} \lhd \la$ and $cont(I|J)+\gamma_{I|J}=\mu$.

The primitive vector $\pi_{I|J}$ has a unique leading term $vy_{I|J}$ 
and its remaining terms, that we will call lower terms, are of the form $w_{K|L}y_{K|L}$, where $ht(K|L)<ht(I|J)$. Hence $\pi_{I|J}=vy_{I|J}+$ lower terms. Note that this agrees with our earlier definition of lower terms for $\pi_{ij}$.

\begin{lm}\label{LI} The elements $\pi_{I|J}$ for admissible $(I|J)$ are linearly independent.
\end{lm}
\begin{proof}
It is enough to consider vectors $\pi_{K|L}$ of the same weight $\la_{K|L}=\la_{I|J}$. These correspond to admissible $(K|L)$ such that $cont(K|L)=cont(I|J)$. It is clear that
$ht(K|L)=ht(I|J)$ as well. The leading terms of these $\pi_{K|L}$ are $vy_{K|L}$ and all remaining terms have heights lower than $ht(I|J)$. 
Since elements $vy_{K|L}$ for different admissible $(K|L)$ are linearly independent, arguing modulo terms of heights lower than $ht(I|J)$ we conclude that 
elements $\pi_{K|L}$ are linearly independent.
\end{proof}

Using the last lemma we can construct a number of $G_{ev}$-primitive elements of $F_k$. This number is maximal whenever $\la$ is 
$(I|J)$-robust since every $\pi_{I|J}$ for admissible $(I|J)$ is a $G_{ev}$-primitive element of $F_k$. The next theorem shows that, in a case of characteristic zero, for a large 
class of weights, the vectors $\pi_{I|J}$ form a basis of $G_{ev}$-primitive vectors in $F_k$ of weight $\la_{I|J}$.

\begin{teo}\label{Fwedge}
Assume that $char(K)=0$, $(I|J)=(i_1\ldots i_k|j_1\ldots j_k)$ is admissible, $\la$ is $(I|J)$-robust, $\la_{I|J}=\sigma$, and $\sigma_m\geq n$.
Then the set of all vectors $\pi_{K|L}$ for admissible $(K|L)$ such that $cont(K|L)=cont(I|J)$ form a basis of the set of $G_{ev}$-primitive vectors of weight $\sigma$ in $F_k$.
\end{teo}
\begin{proof} 
Since $\lambda_m\geq \sigma_m$, the assumption $\sigma_m\geq n$ gives $\lambda_m\geq n$. Since $char(K)=0$, this implies $\la$ is typical and $H^0_G(\la)$ is irreducible. 
Theorem 6.11 of \cite{br} states that  
$H^0_G(\la)$ is a direct sum of $G_{ev}$-modules $H^0_{G_{ev}}(\tau)$, where $\tau=(\mu|\nu)$ is such that $\mu<\la$, appearing with multiplicity equal to the Littlewood-Richardson coefficient $C^{\lambda'}_{\mu' \nu}$ corresponding to the transposed partition $\lambda'$ of $\la$, the transposed partition $\mu'$ of $\mu$ and the partition $\nu$. 
Therefore the dimension of the set of $G_{ev}$-primitive vectors of weight $\sigma=\la_{I|J}=(\mu|\nu)=\tau$ in $F_k$ is the multiplicity 
$C^{\lambda'}_{\mu' \nu}$ of $H^0_{G_{ev}}(\tau)$ in the above decomposition of $H^0_G(\la)$.
The number $C^{\lambda'}_{\mu' \nu}$ equals the number of Littlewood-Richardson tableaux of shape $\lambda'/\mu'$ and of weight $\nu$.

Consider the case when $\tau=(\mu|\nu)=\sigma=\la_{I|J}$. From the assumption of the theorem we infer that the shape of the skew tableau $\lambda'/\mu'$ is as follows.
The entries in the first $n$ rows start at the $(m+1)$st positions. Since $\la$ is $(I|J)$-robust and $\sigma_m\geq n$, the remaining entries of the skew tableuax
are positioned under the $n$th row and in the first $m$ columns in such a way that every row contains no more than one entry of the skew tableau.

Every Littlewood-Richardson tableau $\lambda'/\mu'$ must have all entries in the first row equal to $m+1$, in the second row equal to $m+2$ and so on until the entries in the $n$-th row 
(if any) equal to $m+n$. This follows because this tableau is semistandard, the entries in the first $n$ rows start at the same position and the entry at the end of the $j$-th row  (if any) equals $m+j$.
Since no two remaining entries lie in the same row, such tableau is semistandard if and only if all the entries in the same column are weakly increasing. Since $\la$ is $(I|J)$-robust, 
we infer that any such tableau produces a lattice permutation (there are too few remaining entries for their order to matter).
Therefore Littlewood-Richardson tableau of shape $\lambda'/\mu'$ and content $\nu$ are in one-to-one correspondance to admissible $(K|L)$ of the same content as $(I|J)$.

Lemma \ref{LI} shows that $\pi_{K|L}$ are linearly independent, hence they form a basis of the set of $G_{ev}$-primitive vectors of weight $\sigma$ in $F_k$.

\end{proof} 

The above theorem illustrates that elements $\pi_{I|J}$ can be used to describe all $G_{ev}$-primitive vectors of $H^0_G(\la)$ of certain weights. 

For small values of $m$ and $n$, the elements $\pi_{I|J}$ capture a lot of $G_{ev}$-primitive vectors. For example, if $\la$ is a restricted hook weight of $G=GL(2|2)$ 
with $\la_2\geq 2$, then according to \cite{gm1}, vectors $\pi_{I|J}$ belong to $H^0_G(\la)$ except when $\la_1=\la_2$ or $\la_3=\la_4$, in which case $\pi_{13|24}$ and $\pi_{14|23}$ do not belong to $H^0_G(\la)$. However, if $\la_1=\la_2$, then $-\pi_{13|24}-\pi_{14|23}$ ($=m_2$ in the notation 
of \cite{gm1}) belongs to $H^0_G(\la)$; and if $\la_4=\la_4$, then $-\pi_{13|24}+\pi_{14|23}$ ($=n_2$ in the notation of \cite{gm1}) belongs to $H^0_G(\la)$.
This suggests that even when $\pi_{I|J}$ do not belong to $H^0_G(\la)$, some of their linear combinations produce $G_{ev}$-primitive elements from $H^0_G(\la)$.
It is not clear if all $G_{ev}$-primitive elements of $H^0_G(\la)$ are obtained this way.

Other applications of $G_{ev}$-primitive vectors $\pi_{I|J}$ that are in $H^0_G(\la)$ are in the situation analogous to Corollary \ref{cor} when we can decide about the irreducibility of $H^0_G(\la)$ using a map $\phi_k$ defined below.
Finally, in the next section we show an application of $G_{ev}$-primitive vectors $\pi_{ij}$ to the linkage principle for $G$.

\subsection{The map $\phi_r$}

Consider a map $\phi_r^{\otimes}:F_r^{\otimes} \to F_r$ defined by
\[\phi_r^{\otimes}(w\otimes y_{i_1 j_1}\otimes \ldots \otimes y_{i_r j_r})=(w)_{i_1j_1}D\ldots _{i_rj_r}D\]
for each $w\in H^0_{G_{ev}}(\la)$ and indices $1\leq i_1,\ldots, i_r\leq m$, $m+1 \leq j_1,\ldots, j_r\leq m+n$. 

\begin{lm}\label{factor}
The map $\phi_r^{\otimes}$ induces the map $\phi_r:F_r \to F_r$ defined by
\[\phi_r(w\wedge y_{i_1 j_1}\wedge \ldots \wedge y_{i_r j_r})=(w)_{i_1j_1}D\ldots _{i_rj_r}D\]
for each $w\in H^0_{G_{ev}}(\la)$ and indices $1\leq i_1,\ldots, i_r \leq m$ and $m+1\leq j_1, \ldots, j_r\leq m+n$.
\end{lm}
\begin{proof}
To see that the map $\phi_r$ is well defined, it is enough to verify that odd superderivations $_{i_1j_1}D$ and $_{i_2j_2}D$, where $1\leq i_1,i_2 \leq m$ and $m+1\leq j_1,j_2 \leq m+n$, supercommute. The reader is asked to verify first that each $_{i_1j_1}D_{i_2j_2}D+ _{i_2j_2}D_{i_1j_1}D$ annihilates additive generators $c_{K|L}$ of $A(m|n)$ by considering two separate cases when $i_1=i_2$ and $i_1\neq i_2$. Afterwards, apply the quotient rule on elements $\frac{a}{b}$, where $a,b\in A(m|n)$ and $b$ is even, to complete the argument.
\end{proof}

\begin{lm}\label{morr}
The map $\phi_r$ is a morphism of $G_{ev}$-modules and its image is the intersection of the simple module $L_G(\la)$ with the $r$-th floor $F_r$ of $H^0_G(\la)$.
\end{lm} 
\begin{proof}
We will proceed by induction and check that $\phi_r$ preserves the action of even superderivations $_{kl}D$.
The initial step for $r=1$ is proven in Lemma \ref{mor1}. Assume the statement is true for $r=s$ and we will prove it for $r=s+1$.

Assume either $1\leq k,l \leq m$ or $m+1\leq k,l \leq m+n$ and $r>1$. For short, we will write $wy_{i_1 j_1}\ldots y_{i_t j_t}$ 
for expressions like $w\otimes (y_{i_1 j_1} \wedge \ldots y_{i_t j_t})$.

Then $(\phi_{s+1}(wy_{i_1 j_1}\ldots y_{i_s j_s}y_{ij}))_{kl}D = (w)_{i_1 j_1}D\ldots _{i_s j_s}D_{ij}D_{kl}D $. On the other hand, Lemmas \ref{1''} and \ref{1'''} imply
\[\begin{array}{ll}(wy_{i_1 j_1}\ldots y_{i_s j_s}y_{ij})_{kl}D=&(wy_{i_1 j_1} \ldots y_{i_s j_s})_{kl}D y_{ij} \\
&-\delta_{li} wy_{i_1 j_1}\ldots y_{i_s j_s} y_{kj}+\delta_{jk} wy_{i_1 j_1}\ldots y_{i_s j_s} y_{il}.\end{array}\]

Since $\phi_s((wy_{i_1 j_1} \ldots y_{i_s j_s})_{kl}D)=\phi_s(wy_{i_1 j_1} \ldots y_{i_s j_s})_{kl}D$ by the inductive assumption,
we compute 
\[\begin{array}{ll}\phi_{s+1}((wy_{i_1 j_1}\ldots y_{i_s j_s}y_{ij})_{kl}D)=&(w)_{i_1 j_1}D\ldots _{i_s j_s}D_{kl}D_{ij}D\\
&-\delta_{li}  (w)_{i_1 j_1}D\ldots _{i_s j_s}D_{kj}D \\
&+\delta_{jk} (w)_{i_1 j_1}D\ldots _{i_s j_s}D_{il}D.\end{array}\]

Since $_{ij}D_{kl}D-\,_{kl}D_{ij}D=\delta_{jk} \, _{il}D -\delta_{li} \, _{kj}D$, we conclude that 
\[(\phi_{s+1}(wy_{i_1 j_1}\ldots y_{i_s j_s}y_{ij}))_{kl}D =\phi_{s+1}((wy_{i_1 j_1}\ldots y_{i_s j_s}y_{ij})_{kl}D).\] 

For the second part of the statement, use Poincare-Birkhoff-Witt theorem and order the superderivations in the following way. Start with $_{ij}D$ and $_{ij}^{(r)}D$ with $i>j$, continue with odd $_{ij}D$ with $i<j$ followed by even $_{ij}D$, $_{ij}^{(r)}D$ for $i<j$ and $\binom{_{ii}D}{r}$.
The highest vector $v$ of $L_G(\la)$ is annihilated by $_{ij}D$ and $_{ij}^{(r)}D$ where $i>j$. Each odd superderivation $_{ij}D$, where $i<j$, sends a 
vector from floor $F_k$ to the next floor $F_{k+1}$ and each even $_{ij}D$, $_{ij}^{(r)}D$ for $i<j$ and $\binom{_{ii}D}{r}$ sends $F_k$ to itself. To land in the $r$-floor, exactly $r$ different odd superderivations $_{i_t j_t}D$  with $i_t<j_t$ must be applied to $v$. Therefore $L_G(\la)\cap F_r$ is the sum of $G_{ev}$-modules generated by $(v)_{i_1 j_1}D \ldots _{i_r j_r}D$, where 
$1\leq i_t \leq m$ and $m+1\leq j_t \leq m+n$ for each $t=1, \ldots r$. Since $\phi_r$ is a $G_{ev}$-morphism, we conclude that this is exactly the image of the map $\phi_r$.
\end{proof}

\section{Comments about blocks and linkage principle for $\GL(m|n)$}

In this section we assume that the base field $K$ has characteristic $p>2$.

\subsection{Blocks}
Define the blocks of $G$ as follows.
Indecomposable injective $G$-supermodules $I(\la)$ and $I(\mu)$ occur in the same block of $G$ if and only if there is a chain of indecomposable injective supermodules
$I(\la)=I(\la_1), \ldots , I(\la_r)=I(\mu)$ such that every pair of consecutive injective supermodules $I(\la_j)$ and $I(\la_{j+1})$ contain a common simple 
composition factor.
If the weights $\la$ and $\mu$ belong to the same block, we call them linked and denote $\la\sim \mu$.

A $d$-exponent $d(\la)\geq 0$ of a dominant weight $\la$ of $\GL(m)$ is the maximal number such that $\la_i-\la_{i+1} \equiv -1 \pmod {p^{d(\la)}}$
for every $i=1, \ldots, m-1$.

There is the following classical result of Donkin (see \cite{don}, \S 1 (6)) describing the structure of blocks for the general linear group $\GL(m)$.

\begin{pr}\label{Donkin}
Dominant polynomial weights $\la$ and $\mu$ of $\GL(m)$ belong to the same block of $\GL(m)$ if and only if $d(\la)=d(\mu)=d$ and 
there is a permutation $\sigma$ of the set $\{1, \ldots, m\}$ such that 
\[\la_i-i\equiv \mu_{\sigma(i)}-\sigma(i)\pmod {p^{d+1}}\]
for every $i=1, \ldots, m$.
\end{pr}

Say that $\la$ is strongly linked to $\mu$ if there is a chain of weights $\la=\la_1, \ldots, \la_r=\mu$ such that 
for every consecutive pair of weights $\la_j$ and $\la_{j+1}$ we have either 
$L_G(\la_j)$ is a composition factor of $H^0_G(\la_{j+1})$ or $L_G(\la_{j+1})$ is a composition factor of $H^0_G(\la_j)$.

Say that $\la$ is weakly linked to $\mu$ if there is a chain of weights $\la=\la_1, \ldots, \la_r=\mu$ such that 
for every consecutive pair of weights $\la_j$ and $\la_{j+1}$ we have either 
$Ext^1_G(L_G(\la_j), L_G(\la_{j+1}))\neq 0$ or $Ext^1_G(L_G(\la_{j+1}), L_G(\la_{j}))\neq 0$.

The next statement shows equivalence of linkage, strong linkage and weak linkage.

\begin{lm}\label{strong}
Dominant weights $\la$ and $\mu$ belong to the same block of $\GL(m|n)$ if and only if they are strongly linked if and only if they are weakly linked. 
\end{lm}
\begin{proof}
Assume $L_G(\la)$ is a composition factor of $H^0_G(\mu)$. Since $L_G(\mu)$ is the simple socle of $H^0_G(\mu)$, $L_G(\la)$ is a composition factor in $I(\mu)$ and $\la\sim \mu$.
This shows that the strong linkage implies the linkage.

If $L_G(\la)$ is a composition factor in $I(\mu)$, then there is a chain 
of simple supermodules $L_G(\la)=L_G(\la_1), \ldots, L_G(\mu)$ 
such that for each consecutive pair $L_G(\la_j)$ and $L_G(\la_{j+1})$ we have $Ext^1_G(L_G(\la_j), L_G(\la_{j+1}))\neq 0$. 
Therefore the linkage implies the weak linkage.

Finally, assume that $Ext^1_G(L_G(\la),L_G(\mu))\neq 0$. Since $Ext^1_G(L_G(\la), L_G(\mu))\simeq Ext^1_G(L_G(\mu), L_G(\la))$ using contravariant duality given by supertransposition 
$\tau$, we can assume that $\mu \ntriangleright \la$. Repeating the arguments in \cite{jan}, II, 2.14 we get $Hom_G(rad_G V(\la),L_G(\mu))\neq 0$ and $L_G(\mu)$ is a composition factor in $V_G(\la)$. Using contravariant duality again we conclude that $L_G(\mu)$ is a composition factor in $H^0_G(\la)$. Therefore the weak linkage implies the strong linkage.
\end{proof}

\subsection{Even linkage}

We say that dominant weights $\la$ and $\mu$ of $\GL(m|n)$ are even-linked and denote it by $\la\sim_{ev} \mu$ if and only if 
they belong to the same $G_{ev}$-block. This happens if and only if $\la^+\sim \mu^+$ are linked with respect to $\GL(m)$ and 
$\la^-\sim \mu^-$ are linked with respect to $\GL(n)$. 
In what follows we will denote by $d^+$ the $d$-exponent of 
$\la^+$ with respect to $\GL(m)$ and by $d^-$ the the $d$-exponent of $\la^-$ with respect to $\GL(n)$.
Since $G_{ev}\cong \GL(m)\times \GL(n)$, the blocks of group $G_{ev}$ are built from blocks of groups $\GL(m)$ and $\GL(n)$. 

The next statement shows that even linkage implies linkage. 

\begin{lm}
If $\la\sim_{ev} \mu$, then $\la\sim \mu$. 
\end{lm}
\begin{proof}
The analogue of Lemma \ref{strong} is valid for $G_{ev}$ and it allows us to assume that $L_{G_{ev}}(\mu)$ is a composition factor of $H^0_{G_{ev}}(\la)$. 

We will view $G_{ev}$-modules as $P$-supermodules, where $P$ is a parabolic supergroup of $G$ corresponding to $G_{ev}$ and Borel subsupergroup $B$, 
via an epimorphism $P\to G_{ev}$. In particular, the supergroup $B$ acts via $B\to B_{ev}$ and the highest weight vector of $L_{G_{ev}}(\mu)$ is also the highest 
weight vector with respect to the $B$-action.  

Denote $M(\mu)=ind_P^G(L_{G_{ev}}(\mu))$. Then $L_{G_{ev}}(\mu)$ is a $G_{ev}$-submodule of $M(\mu)$, which is isomorphic to $L_{G_{ev}}(\mu)\otimes S(Y)$.
Therefore the highest weight vector of $L_{G_{ev}}(\mu)$ generates a $G$-supersubmodule of $M(\mu)$ which is isomorphic to $L_G(\mu)$.

According to \cite{z}, $G/P$ is an affine superscheme, and therefore the functor $ind_P^G$ is exact (even faithfully exact). Thus a composition series of $H^0_{G_{ev}}(\la)$ 
induces a filtration of $H^0_G(\la)=ind_P^G(H^0_{G_{ev}}(\la))$. Since the supermodule $M(\mu)$ is a factor in this filtration of $H^0_G(\la)$ and it contains $L_G(\mu)$, 
the module $L_G(\mu)$ is a composition factor in $H^0_G(\la)$.
\end{proof}

\subsection{The canonical alcove}

Next, we will consider the simplest case when the weight $\la$ is such that $H^0_{G_{ev}}(\la)$ is irreducible $G_{ev}$-module. In this case we are able to describe all weights of the first floor $F_1$ that are linked to $\la$.

Let $\rho$ be the weight of $G$ that is equal to the half sum of positive even roots of $G$ minus the half sum of its positive odd roots. Denote by $(.,.)$ the bilinear form on the set of weights $X(T)$ of $G$ defined by $(\epsilon_i,\epsilon_j)=\delta_{ij} (-1)^{|\epsilon_i|}$. For more on the bilinear product and roots in the superalgebra setting see \cite{kac} or \cite{cw}. Denote by $W_{af}$ the affine Weyl group of $G$ and consider its dot action on weights of $G$ given by $w\cdot \la = w(\la+\rho)-\rho$ for $w\in W_{af}$. 
The standard (canonical) alcove $C$ of $G$ is given as
\[C=\{\la\in X(T)\otimes_{\mathbb Z} \mathbb R| 0< (\la+\rho, \beta)<p \text{ for all } \beta \text{ that are positive coroots of } G\}.\]
It is classical result that for the closure $\overline{C}$ of $C$ the set $\overline{C}\cap X(T)$ is a fundamental domain for $W_{af}$ acting on $X(T)$ (see II.6.2 of \cite{jan}).
Since all composition factors $L_{G_{ev}}(\mu)$ of $H^0_{G_{ev}}(\la)$ satisfy $\mu \lhd \la$ and $C$ is the lowest alcove, we infer that $H^0_{G_{ev}}(\la)$ is irreducible if $\la\in C$.

\begin{lm} \label{canon}
Let $\la\in C$ and $k>0$. 
Let $\mu$ be a weight of $F_k$ and $Z$ be the $G_{ev}$-submodule of $F_k$ generated by all elements of $F_k$ of weight bigger than $\mu$. Then the
codimension of $Z_{\mu}$ in $(F_k)_\mu$ is not bigger than the number of admissible $(I|J)$ such that $\la_{I|J}=\mu$.
\end{lm}
\begin{proof}
The space $(F_k)_{\mu}$ is a span of vectors of weight $\mu$ of type $zy_{K|L}$ where $z\in H^0_{G_{ev}}(\la)$, and
the weight $\kappa \unlhd \la$ of $z$ and the weight $\nu$ of $y_{K|L}$ satisfy $\kappa+\nu=\mu$. 

Assume $\kappa\neq \la$ and write $z\in H^0_G(\la)$ as $z=v\zeta$, where $\zeta\in \Dist(G_{ev})$.
This is possible because $H^0_{G_{ev}}(\la)$ is an irreducible $G_{ev}$-module.
Then $vy_{K|L}\in F_k$ and its weight $\gamma$ satisfies $\mu \lhd \gamma$. Therefore $vy_{K|L}\in Z$.
Apply $\zeta$ to get $(vy_{K|L})\zeta= zy_{K|L}$+ terms $w_{M|N} y_{M|N}$, where $w_{M|N}\in H^0_{G_{ev}}(\la)$ and 
$ht(y_{M|N})<ht(y_{K|L})$. 
Using induction on the height of $y_{K|L}$ we conclude that the codimension of $Z_{\mu}$ in $(F_k)_\mu$ is not bigger than 
the number of admissible $(I|J)$ such that $\la_{I|J}=\mu$.
\end{proof}

\begin{cor}\label{baza}
Assume $\la\in C$ and a weight $\mu$ is such that $\pi_{I|J}\in H^0_G(\la)$ for all admissible $(I|J)$ such that $\la_{I|J}=\mu$.  
Then such elements $\pi_{I|J}$ form a basis of $G_{ev}$-primitive vectors of weight $\mu$ in $H^0_G(\la)$.
\end{cor}
\begin{proof}
Use Lemma \ref{LI}.
\end{proof}

\begin{lm}\label{Nakayama1}
Assume $char(K)=p>2$. If the weights $\la_{ij}$ and $\la_{kl}$ are $G_{ev}$-linked, then 
$\la^+_i-i\equiv \la^+_k-k \pmod p$ and $\la^-_j-j\equiv \la^-_l-l \pmod p$. 
Consequently, if $\la_{ij}$ and $\la_{kl}$ are $G_{ev}$-linked, then $\omega_{ij}\equiv 0 \pmod p$
if and only if $\omega_{kl}\equiv 0 \pmod p$.
\end{lm}
\begin{proof}
We have $\la_{ij}=(\la^+_1, \ldots , \la^+_i-1, \ldots , \la^+_m|\la^-_1, \ldots, \la^-_j+1, \ldots, \la^-_n)$ and 
$\la_{kl}=(\la^+_1, \ldots , \la^+_k-1, \ldots , \la^+_m|\la^-_1, \ldots, \la^-_l+1, \ldots, \la^-_n)$. We will show that
if $\mu=(\la^+_1, \ldots , \la^+_i-1, \ldots , \la^+_m)$ and $\nu=(\la^+_1, \ldots , \la^+_k-1, \ldots , \la^+_m)$ are $GL(m)$-linked,
then $\la^+_i-i\equiv \la^+_k-k \pmod p$; the proof of the claim about $\la^-_j-j\equiv \la^-_l-l \pmod p$ is analogous.

If $i=k$, then $\la^+_i-i\equiv \la^+_k-k \pmod p$ is trivial. Assume $i\neq k$. If $\mu \sim \nu$, then by Proposition \ref{Donkin} there
is a permutation $\sigma$ of the set $\{1, \ldots, m\}$ such that 
$\mu_j-j\equiv \nu_{\sigma(j)}-\sigma(j)\pmod {p}$. Since $\la^+_i-1-i \not\equiv \la^+_i-i \pmod{p}$, we must have $\sigma(i)\neq i$. Decompose 
$\sigma$ into a product of disjoint cycles and denote the cycle containing $i$ by $\tau=(i_1, \ldots , i_s)$. If $\tau$ does not contain $k$, 
$\sum_{t=1}^s \mu_{i_s} = \sum_{t=1}^s \nu_{i_s} -1$ contradicting $\sum_{t=1}^s (\mu_{i_s} - i_s) \equiv \sum_{t=1}^s (\nu_{i_s} - i_s) \pmod p$.
Therefore we can assume that $i_1=i$ and $i_t=k$. There is a sequence of congruences modulo $p$:
$\la^+_i-1-i=\mu_i-i\equiv \nu_{i_2}-i_2 = \mu_{i_2}-i_2 \pmod p$, 
$\mu_{i_2}-i_2\equiv \nu_{i_3}-i_3 =\mu_{i_3}-i_3 \pmod p$, \ldots, 
$\mu_{i_{t-1}}-i_{t-1} \equiv \nu_k -k =\la^+_k-1-k \pmod p$. By transitivity we conclude $\la^+_i-i\equiv \la^+_k-k \pmod p$.

The second part of the Lemma follows from the definition of $\omega_{ij}$ and $\omega_{kl}$.
\end{proof}

Note that if $\la$ is atypical and $w\in W_{af}$, then $w\cdot \la$ is again atypical.

\begin{teo}\label{link1} Assume the weights $\la$ and $\la_{ij}$ are dominant. If $char(K)=0$, then the simple module $L_G(\la_{ij})$ 
is a composition factor of $H^0_G(\la)$ if and only if $\omega_{ij}=0$.
If $char(K)=p>2$ and $\la\in C$, then the module $L_G(\la_{ij})$ is a composition factor of $H^0_G(\la)$ if and only if $\omega_{ij}\equiv 0 \pmod p$. 
In this case, in an appropriate
subquotient of $H^0_G(\lambda)$, the image of the primitive vector $\pi_{ij}$ is the highest weight vector for the composition factor $L_G(\la_{ij})$.
\end{teo}
\begin{proof}
Assume first $char(K)=0$. If $\omega_{ij}=0$, then the image of $\phi_1$ does not contain the primitive vector $\pi_{ij}$.
Therefore $\pi_{ij}$ generates a submodule of $H^0_G(\la)$ that is not contained in $L_G(\la)$ and has the highest weight $\la_{ij}$.
Hence the simple module $L_G(\la_{ij})$ belongs to the composition series of $H^0_G(\la)$.

The module $H^0_G(\la)$ considered as a $G_{ev}$-module is completely reducible. 
If $\omega_{ij}\neq 0$, then $\pi_{ij}$ belongs to $L_G(\la)$ and 
since $\pi_{ij}$ is the only $G_{ev}$-primitive vector of weight $\la_{ij}$ in $H^0_G(\la)$, 
$L_G(\la_{ij})$ is not a composition factor of $H^0_G(\la)$.

Now assume that $char(K)=p>2$. Denote by $\langle S\rangle_{ev}$ the $G_{ev}$-module generated by the set $S$. If $\omega_{ij}\equiv 0 \pmod p$, then 
$\langle \pi_{ij} \rangle_{ev}$ belongs to the kernel of the map $\phi_1$, and therefore its intersection with $L_G(\la)$ is zero.
The $G$-module $M$ generated by $v$ and $\pi_{ij}$ is a submodule of $H^0_G(\la)$ containing $L_G(\la)$. Since the factormodule $M/L_G(\la)$ 
has the highest vector $\pi_{ij}$, this implies that $L_G(\la_{ij})$ belongs to the composition series of $H^0_G(\la)$.

Further, denote by $M_1$ the $G$-subsupermodule of $H^0_G(\la)$ generated by all vectors of weights $\la$ and 
$\la_{kl}$ for $\omega_{kl}\equiv 0 \pmod p$ and denote by $M_1^{ev}$ the $G_{ev}$-submodule of $F_1$ generated by all vectors of weights 
$\la_{kl}$ for $\omega_{kl}\equiv 0 \pmod p$. Since $L_G(\la)$ is a socle of $H^0_G(\la)$, it is also a socle of $M_1$.

We will show that $M_1/L_G(\la) \cap F_1 \cong M_1^{ev}/(L_G(\la)\cap F_1)$ as $G_{ev}$-modules.
Use Poincare-Birkhoff-Witt theorem and order the superderivations in the following way. Start with odd $_{ab}D$ where $a>b$, continue with odd $_{ab}D$ with $a<b$ followed by 
even $_{ij}D$, $_{ij}^{(r)}D$ and $\binom{_{ii}D}{r}$. Any composition of odd $_{ab}D$ where $a>b$ applied to generators of $M_1$ gives elements of $F_0$, 
which belong
to $L_G(\la)$ because $\la\in C$. Any composition of odd $_{ab}D$ where $a<b$ applied to generators of $M_1$ either belongs to $L_G(\la)$ 
or belongs to higher floors $F_k$ for $k>1$.
Since application of any composition of even $_{ij}D$, $_{ij}^{(r)}D$ and $\binom{_{ii}D}{r}$ preverves each floor $F_k$, the module 
$M_1/L_G(\la) \cap F_1$ is generated by $G_{ev}$-action on vectors of weight $\la_{kl}$ such that $\omega_{kl}=0$. Therefore 
$M_1/L_G(\la) \cap F_1 \cong M_1^{ev}/(L_G(\la)\cap F_1)$.

Using this isomorphism we infer that if $L_G(\la_{ij})$ is a composition factor of $M_1/L_G(\la)$, then $\la_{ij}$ is $G_{ev}$-linked to some 
$\la_{kl}$ such that $\omega_{kl}\equiv 0 \pmod p$.
In this case we have $\omega_{ij}\equiv 0 \pmod p$ by Lemma \ref{Nakayama1}.

Assume that $L_G(\la_{ij})$ is a composition factor of 
$H^0_G(\la)/M_1$ such that $\omega_{ij}\not\equiv 0 \pmod p$ and a weight $\la_{ij}$ is maximal with such property.
Denote by $Z_1$ the $G_{ev}$-submodule of $F_1$ spanned by all vectors of weight $\mu \rhd \la_{ij}$ in $F_1$ such that 
$\mu\neq \la_{ij}$. 
By Lemma \ref{canon} the codimension of $(Z_1)_{\la_{ij}}$ in $(F_1)_{\la_{ij}}$ equals one. 
Moreover, $\pi_{ij}\notin Z_1$.
Since $\omega_{ij}\not\equiv 0 \pmod p$, the vector $\pi_{ij}\in L_G(\la)\subset M_1$. Additionally, $Z_1\subset M_1$ and therefore 
$(M_1)_{\la_{ij}}=(F_1)_{\la_{ij}}$, which is a contradiction with $L_G(\la_{ij})$ being a composition factor of $H^0_G(\la)/M_1$. 
Therefore $L_G(\la_{ij})$ is in the composition series of $H^0_G(\la)$ only if $\omega_{ij}\equiv 0 \pmod p$.   
\end{proof}

The composition factors of $M(\la)$ are described as follows.

\begin{cor}\label{m1}
Assume $\la\in C$ and a weight $\la_{ij}$ is dominant. 
The simple supermodule $L_G(\la_{ij})$ is a composition factor of $M(\la)$ if and only if $\pi_{ij}\in M(\la)$ and $\omega_{ij}\equiv 0 \pmod p$.
\end{cor}

Note that Lemma 3.1 of \cite{fm} shows that $\pi_{ij}$ need not belong to $M(\la)$, hence the assumption $\pi_{ij}\in M(\la)$ is necessary.

The last Theorem shows a partial case when strong linkage of $\la_{ij}$ to $\la$ is related to vanishing of elements $\omega_{ij}$ (modulo $p$). 
Interestingly, this property is invariant under even linkage due to Lemma \ref{Nakayama1}.

\subsection{Odd linkage of $\la$ and $\la_{I|J}$}

Theorem \ref{link1} motivates the definition of odd linkage. 

Let $\la$ be a dominant weight of $\GL(m|n)$ and $(I|J)=(i_1 \ldots i_r|j_1 \ldots j_r)$ be admissible. 
Define the concept of odd linkage of weight $\la$ and $\la_{I|J}$ as follows. 
We say $\la$ and $\la_{I|J}$ are odd-linked, and write $\la\sim_{odd} \la_{I|J}$, if and only if 
the following condition $C(I|J)$ is satisfied.

There is an rearrangement $(I'|J')$ (not necessarily admissible) of $(I|J)$ such that $\la_{I|J}=\la_{I'|J'}$ and
\[\begin{aligned}&\omega_{i'_1,j'_1}(\la)\equiv 0 \pmod p, \quad \omega_{i'_2,j'_2}(\la_{i'_1,j'_1})\equiv 0 \pmod p, \ldots, \\ 
&\omega_{i'_s,j'_s}(\la_{i'_1\ldots i'_{s-1}| j'_1,\ldots, j'_{s-1}})\equiv 0 \pmod p. \end{aligned}\]

We allow $(I|J)=\emptyset$, which gives $\la\sim_{odd} \la$.

Let us formulate the following consequence of Proposition \ref{Donkin}.
\begin{pr}\label{Nakayama}
Let $char(K)=p>2$, $\la$ be dominant and $(I|J)$ and $(K|L)$ be admissible.
If $\la_{I|J}$ and $\la_{K|L}$ are $G_{ev}$-linked, then condition $(C_{I|J})$ implies $(C_{K|L})$. 
\end{pr}
\begin{proof}
It is clear that $(I|J)$ and $(K|L)$ are of the same length $r$. By Proposition \ref{Donkin},
there is a bijection $\sigma^+:I\to K$ such that 
\[\la^+_{i_s}-i_s\equiv \la^+_{\sigma^+(i_s)}- \sigma^+(i_s) \pmod p\]
for each $s=1, \ldots r$ 
and there is a bijection $\sigma^-:J\to L$ such that 
\[\la^-_{j_t}-j_t\equiv \la^+_{\sigma^-(j_t)}- \sigma^-(j_t) \pmod p\]
for each $t=1, \ldots r$.
Using the definition of $\omega_{a,b}$ we verify that the rearrangement $(K'|L')$ of $(K|L)$ given by 
$k'_s=\sigma^+(i_s)$ and $l_t=\sigma^-(j_t)$ for $s,t=1, \ldots, r$ satisfies the condition
$(C_{K|L})$.
\end{proof}

This shows that odd and even linkage work together remarkably well. 

\subsection{Linkage conjecture}

\begin{pr}\label{Link}
Assume the weight $\la_{I|J}$ is dominant and $(I|J)$ is admissible. 
The weights $\la_{I|J}$ and $\la$ are linked if
there is a sequence $\alpha_1, \alpha_2, \ldots \alpha_r$ of simple odd roots of $\GL(m|n)$ (not necessarily distinct!)
and an element $w\in W_{af}$ satisfying the following properties:

$\la^{(0)}=\la$,

$\la^{(i+1)}=\la^{(i)} +\alpha_{i+1}$ and $(\la^{(i)}+\rho, \alpha_{i+1})=0$ for $i=0, \ldots r-1$, and 

$w\cdot \la^{(r)} = \la_{I|J}$.

\end{pr}
\begin{proof}
Reformulate previous results by replacing $\omega_{ij}$ using the language of the bilinear form $(.,.)$.
\end{proof}

Proposition \ref{Nakayama} shows that applying even linkage at any floor during the process of building $\la_{I|J}$ 
does not change the specific set of linked weights in Proposition \ref{Link}. 
Therefore we ask whether the condition in Proposition \ref{Link} is also necessary for the linkage of weights. 
We suspect that this might be the case.

\bigskip

{\it Acknowledgement.} The author would like to express his appreciation for the work of the referee and thanks for valuable comments 
that helped to significantly improve the presentation and readability of this paper.

\end{document}